\documentclass[12pt,reqno]{amsart}
\usepackage{amsmath, amsthm, amssymb, stmaryrd}
\usepackage{hyperref}
\usepackage{enumerate}
\usepackage{url}

\let\OLDthebibliography\thebibliography
\renewcommand\thebibliography[1]{
  \OLDthebibliography{#1}
  \setlength{\parskip}{0pt}
  \setlength{\itemsep}{0pt plus 0.3ex}
}

\usepackage{mathtools}

\usepackage{multirow, array}
\usepackage{placeins}
\usepackage{xcolor}

\advance \topmargin by -\headheight
\advance \topmargin by -\headsep

\evensidemargin \oddsidemargin
\newtheorem{thm}{Theorem}[section]
\newtheorem{lemma}[thm]{Lemma}
\newtheorem{prop}[thm]{Proposition}
\newtheorem{cor}[thm]{Corollary}

\theoremstyle{definition}
\newtheorem{defn}[thm]{Definition}

\theoremstyle{remark}
\newtheorem{remark}[thm]{Remark}

\numberwithin{equation}{section}

\newcommand{\vv}[1]{\mathbf{#1}}

\newcommand*\wrapletters[1]{\wr@pletters#1\@nil}
\def\wr@pletters#1#2\@nil{#1\allowbreak\if&#2&\else\wr@pletters#2\@nil\fi}

\def\alp{{\alpha}} 
\def\bet{{\beta}}  
 
\def\del{{\delta}} \def\Del{{\Delta}}

\def\kap{{\kappa}}
\def\lam{{\lambda}} \def\Lam{{\Lambda}}

\def\ome{{\omega}}  
\def\eps{\varepsilon} \def \epsilon {{\varepsilon}}

\def \leq{\leqslant} \def \geq {\geqslant}
\def\le{\leqslant} \def\ge{\geqslant}

\def\d{{\,{\rm d}}}

\def \bH {\mathbb H}

\def \bN {\mathbb N}

\def \bR {\mathbb R}
\def \bZ {\mathbb Z}

\newcommand{\N}{{\Bbb N}}         
\newcommand{\I}{{\Bbb I}}
\newcommand{\R}{{\Bbb R}}        
\newcommand{\Z}{{\Bbb Z}}         

\def \bt {\mathbf t}

\def \bv {\mathbf v}
\def \bx {\mathbf x}

\def \bzero {\mathbf 0}

\def \fe {\mathfrak e}

\def \cB {\mathcal B}
\def \cC {\mathcal C}

\def \cF {\mathcal F}

\def \cI {\mathcal I}

\def \cK {\mathcal K}
\def \cL {\mathcal L}
\def \cM {\mathcal M}

\def \cP {\mathcal P}

\def \cS {\mathcal S}

\def \ord {\mathrm{ord}}
\def \det {\mathrm{det}}

\def \vol {\mathrm{vol}}

\def \supp {{\mathrm{supp}}}

\def \R {{\mathbb{R}}}
\def \Z {{\mathbb{Z}}}
\def \SL {{\mathrm{SL}}}

\def \SL {{\mathrm{SL}}}
\def \Aff {{\mathrm{Aff}}}
\def \dd {{\mathrm{d}}}
\def \SO {{\mathrm{SO}}}
\def \I {{\mathrm{I}}}

\begin{document}
\title[Effective equidistribution for Diophantine approximation]{Effective equidistribution for multiplicative Diophantine approximation on lines}
\author[Sam Chow]{Sam Chow}
\address{Sam Chow, Mathematics Institute, Zeeman Building, University of Warwick, Coventry CV4 7AL, UK}
\email{Sam.Chow@warwick.ac.uk}
\author[Lei Yang]{Lei Yang}
\address{Lei Yang, Institute for Advanced Study, Princeton, New Jersey, 08540, USA}
\email{lyang861028@gmail.com}
\thanks{}
\date{}
\dedicatory{Dedicated to the memory of Marina Ratner}
\begin{abstract}
Given any line in the plane, we strengthen the Littlewood conjecture by two logarithms for almost every point on the line, thereby generalising the fibre result of Beresnevich, Haynes, and Velani. To achieve this, we prove an effective asymptotic equidistribution result for one-parameter unipotent orbits in $\SL(3, \R)/\SL(3,\Z)$. We also provide a complementary convergence statement, by developing the structural theory of dual Bohr sets: at the cost of a slightly stronger Diophantine assumption, this sharpens a result of Kleinbock's from 2003. Finally, we refine the theory of logarithm laws in homogeneous spaces.
\end{abstract}
\maketitle

\section{Introduction}
\label{sec-intro}


\subsection{Multiplicative Diophantine approximation}
\label{subsec-intro-mda-lines}

The Littlewood conjecture (circa 1930) is one of the oldest problems in Diophantine approximation. It asserts that if $\alp, \bet \in \bR$ then
\begin{equation}  \label{littleor} \liminf_{n \to \infty} n \langle n \alpha\rangle  \langle n \beta\rangle = 0, \end{equation}
where for $x \in \bR$ we write $\displaystyle \langle x \rangle = \min_{n \in \bZ} |x - n|$. Despite some remarkable progress---see  \cite{BV, EKL, PV01, V07} and the references within---the Littlewood conjecture remains very much open. For example, we are unable to show that (\ref{littleor}) is valid for the pair $(\sqrt2,\sqrt3)$. On the other hand, from a measure-theoretic point of view Littlewood's conjecture is well-understood. Indeed, if we are only interested in the multiplicative approximation rate of a \emph{typical} point $(\alpha, \beta) \in \R^2$, with respect to the Lebesgue measure on $\R^2$,  a theorem of Gallagher \cite{Gallagher1962} implies the following statement.
\begin{thm}[Gallagher]  \label{reallygal}
For almost every $(\alpha, \beta) \in \R^2$, we have
\begin{equation}
\label{equ:typical-mda-property}
 \liminf_{n \to \infty} n (\log n)^2 \langle n \alpha\rangle  \langle n \beta\rangle = 0.
 \end{equation}
\end{thm}
\noindent In other words, almost surely Littlewood's conjecture holds with a
``log squared'' factor to spare. This is sharp in that
 for any $\kap >2$ the set of $(\alpha  , \beta) \in \R^2$ for which
\[ \liminf_{n \to \infty} n (\log n)^\kap \langle n \alpha\rangle  \langle n \beta\rangle = 0\]
has zero Lebesgue measure \cite{BV2015note, Spe1942}.

In view of Theorem \ref{reallygal}, it is natural to ask the following question: given a planar curve or straight line $C$, does almost every point $(\alpha, \beta) \in C$ satisfy \eqref{equ:typical-mda-property}? The problem was first investigated  by Beresnevich, Haynes, and Velani \cite{Beresnevich-Haynes-Velani2015}, who considered the special case of  vertical lines $L_{\alpha} := \{ (\alpha, \bet):  \bet \in \R \}$. They showed that for any $\alpha \in \R$, almost every point $(\alpha  , \beta) $ on $L_{\alpha}$ satisfies \eqref{equ:typical-mda-property}. They also proved an inhomogeneous version of this statement assuming the truth of the notorious Duffin--Schaeffer conjecture~\cite[\S 1.2.2]{BRV}. Note that in view of Khintchine's theorem \cite[\S 1.2.2]{BRV}, it is easy to deduce that almost every point on $L_{\alpha}$ satisfies
 \begin{equation} \label{Khintchine}
  \liminf_{n \to \infty} n (\log n) \langle n \alpha\rangle  \langle n \beta\rangle = 0.
\end{equation}
Later the first named author \cite{Chow2018} provided an alternative proof of the above mentioned results from \cite{Beresnevich-Haynes-Velani2015}. His method made use of Bohr set technology and generalised unconditionally to the inhomogeneous  setting. This was  subsequently extended to higher dimensions in \cite{CT} and \cite{LittlewoodDS}. In all these results, the fact that the lines under consideration are vertical is absolutely crucial. As a consequence of the framework developed in this paper, we are able to handle arbitrary lines.

\begin{thm}
\label{thm:mda-result}
Let $\cL$ be a line in the plane. Then for almost every $(\alpha, \beta) \in \cL$, with respect to the induced Lebesgue measure on $\cL$, we have
\begin{equation} \label{TwoLogs}
 \liminf_{n \to \infty} n (\log n)^2 \langle n \alpha\rangle  \langle n \beta\rangle = 0.
\end{equation}
\end{thm}

The exponent 2 is sharp, as we now discuss. By symmetry, we may suppose that $\cL$ is defined by
\begin{equation} \label{Lab}
\cL = L_{a,b} = \{ (\alp, \bet) \in \bR^2: \alp = a \bet + b \}
\end{equation}
for some fixed $a,b \in \bR$. The \emph{simultaneous exponent} of $(a,b) \in \bR^2$, denoted $\omega(a,b)$, is the supremum of the set of real numbers $w$ such that, for infinitely many $n \in \bN$, we have
\[
\max\{\langle n a \rangle, \: \langle nb \rangle\} < n^{-w}.
\]
Kleinbock \cite[Corollary 5.7]{Klein2003} showed that if $\omega(a,b) \le 2$ and $\epsilon >0$ then
\[ \liminf_{n \to \infty} n^{1+\epsilon} \langle n \alpha\rangle  \langle n \beta\rangle > 0  \]
for almost every $(\alpha, \beta) \in L_{a,b}$. The \emph{dual exponent} of $(a,b) \in \bR^2$, denoted $\omega^*(a,b)$, is the supremum of the set of real numbers $w$ such that, for infinitely many $(x,y) \in \bZ^2$, we have
\[
\langle xa + yb \rangle \le (|x| + |y|)^{-w}.
\]
With a slightly stronger assumption, we strengthen Kleinbock's result, showing that the exponent 2 in \eqref{TwoLogs} is sharp.
\begin{thm} \label{ConvergenceResult}
Let $(a,b) \in \bR^2$ with $\omega^*(a,b) < 5$, and let $\psi: \bN \to \bR_{\ge 0}$ be a decreasing function such that
\[
\sum_{n =1}^\infty \psi(n) \log n < \infty.
\]
Then for almost all $(\alp, \bet) \in L_{a,b}$ there exist at most finitely many $n \in \bN$ for which
\[
\langle n \alp \rangle  \langle n \bet \rangle < \psi(n).
\]
\end{thm}

By Khintchine transference \cite[Theorem K]{BL2010}, note that if $\omega^*(a,b) < 5$ then $\omega(a,b) < 2$, so our assumption is indeed stronger than Kleinbock's. Our condition is nonetheless typical: using the Hausdorff measure generalisation of the Khintchine--Groshev theorem, namely \cite[Theorem 1.4.37]{BRV}, it can be verified that the exceptional set
\[
\{ (a,b) \in \bR^2: \omega^*(a,b) \ge 5 \}
\]
has Hausdorff dimension $3/2$. We establish Theorem \ref{ConvergenceResult} in \S \ref{ConvergenceTheory} via the methodology of \cite[\S 4]{BV2007} and \cite{BL2007}. To prove the requisite counting lemma, we develop the structural theory of dual Bohr sets, extending the constructions given in \cite{TV2006, TV2008, Chow2018, CT}.

By invoking a recent estimate of Huang and Liu \cite[Theorem 8]{HuangLiu}, we are able to deduce the following variation on Theorem \ref{ConvergenceResult} involving multiplicative Diophantine exponents. The \emph{multiplicative exponent} of $(a,b) \in \bR^2$, denoted $\ome^\times(a,b)$, is the supremum of the set of real numbers $w$ such that, for infinitely many $n \in \bN$, we have
\[
\langle na \rangle  \langle nb \rangle \le n^{-w}.
\]

\begin{thm} \label{ConvergenceVariant}
Let $(a,b) \in \bR^2$ with $\omega^\times(a,b) < 4$, and let $\psi: \bN \to \bR_{\ge 0}$ be a decreasing function such that
\[
\sum_{n =1}^\infty \psi(n) \log n < \infty.
\]
Then for almost all $(\alp, \bet) \in L_{a,b}$ there exist at most finitely many $n \in \bN$ for which
\[
\langle n \alp \rangle  \langle n \bet \rangle < \psi(n).
\]
\end{thm}

The assumption $\ome^\times(a,b) < 4$ also implies Kleinbock's assumption, owing to the trivial inequality $2\ome(a,b) \le \ome^\times(a,b)$, and it is also typical. It follows from the work of Hussain and Simmons \cite[Corollary 1.4]{HussainSimmons}, or alternatively from the prior but weaker conclusions of \cite[Remark 1.2]{BV2015note}, that the exceptional set
\[
\{ (a,b) \in \bR^2: \ome^\times(a,b) \ge 4 \}
\]
has Hausdorff dimension $7/5$.

\begin{remark}
Since the initial release of this manuscript, Huang \cite{Huang} has relaxed the diophantine assumptions of Theorems \ref{ConvergenceResult} and \ref{ConvergenceVariant} to $\ome(a,b) < 2$.
\end{remark}

\bigskip

Our approach to Theorem \ref{thm:mda-result} differs greatly from that of the previous works  \cite{Beresnevich-Haynes-Velani2015, Chow2018, CT}, which were driven by continued fraction or Bohr set analysis. At the heart of our method lies an effective asymptotic equidistribution result for one-parameter unipotent orbits in $\SL(3,\bR)/ \SL(3, \bZ)$. This new theorem in homogeneous dynamics is the subject of the next subsection. 


\subsection{Ratner's equidistribution theorem}
\label{subsec-intro-ratners-theorem}
\par Let $G$ be a Lie group, equipped with a one-parameter unipotent subgroup $U := \{ u(t) : t \in \R\} \subset G$. Let $\Gamma$ be a lattice in $G$, and $X = G/\Gamma$ be the associated homogenous space. Ratner's famous theorem on orbit closure \cite{ratner_2} asserts that for any $x \in X$ the closure of the orbit $Ux$ is a homogeneous subspace of $X$. That is, it has the form $Lx$, where $L$ is an analytic subgroup of $G$ containing $U$ such that (a) the orbit $Lx$ is closed in $X$, and (b) there exists an $L$-invariant probability measure $\mu_L$ supported on $Lx$. This confirmed Raghunathan's topological conjecture; see \cite{Dani1981}. Moreover, Ratner's equidistribution theorem~\cite{ratner_2} tells us that the orbit $Ux$ is equidistributed in $Lx$ in the following sense: for any $F \in C_b(Lx)$, we have
\[
\lim_{T \to +\infty} \frac{1}{T} \int_{0}^{T} F(u(t)x) \dd t = \int_{Lx} F \dd \mu_L.
\]

\par Ratner's equidistribution theorem is a fundamental result in homogeneous dynamics and has many interesting and deep applications to number theory \cite{eskin_mozes_shah1996, eskin1998ICM, DWM, Shah_2, Shah2010, shah2010ICM}. It is based on her seminal work \cite{Ratner} on measure rigidity of unipotent actions, which confirmed Raghunathan's measure conjecture (see \cite{Dani1981}); see also \cite{Margulis-Tomanov} for an alternative proof applicable to algebraic groups. A weakness of Ratner's theorem is that it is not effective: given a particular unipotent orbit, it does not tell how fast it tends to its limit distribution. This renders it less helpful when studying problems that are sensitive to error terms. Establishing Ratner's theorem with an effective error term provides  a more profound viewpoint with regards to the asymptotic behaviour of unipotent orbits in homogeneous spaces, as well as their connections to number theory and representation theory. For this reason, this has been a central topic in homogeneous dynamics ever since Ratner's groundbreaking work in the nineties.

\par For unipotent subgroups which are horospherical, we can establish effective equidistribution using a method from dynamics and results from representation theory, assuming that the ambient group has Kazhdan's property (T) or similar spectral gap properties. This method, called the ``thickening method'', originates in Margulis's thesis \cite{Margulis2004}, and has since been a standard way to study effective equidistribution of horospherical orbits~\cite{Klein-Mar-Effective-Equid, Kleinbock-Shi-Weiss, Dabbs-Kelly-Li, Shi2019}. In particular, for $G = \SL(2,\R)$, since any unipotent subgroup is horospherical, Margulis's thickening method applies. We also refer the reader to \cite{Flaminio-Forni,Strombergsson2004,Strombergsson2013} for direct representation-theoretic approaches to establishing effective equidistribution for unipotent orbits in $\SL(2,\R)/\Gamma$ with explicit error terms.

\par For non-horospherical unipotent orbits, we have effective equidistribution results for $G$ being nilpotent \cite{green_tao2012}, $G= \SL(2,\R) \times \R$ (see \cite{Venkatesh2010,Sarnak-Ubis}), $G = \SL(2, \R) \ltimes (\R^2)^{\oplus k}$ (see \cite{Strombergsson2015, Browning2016, SV2018}), and $G= \SL(2,\R)^k$ (see \cite{Ubis2016}). Their proofs rely on effective equidistribution of unipotent orbits in $\SL(2, \R)/\Gamma$, delicate analysis on the explicit expressions of unitary representations of $\SL(2,\R)$, and Fourier expansions on tori. Thus, one cannot easily adapt the proofs to establish effective equidistribution results for simple Lie groups of higher rank, such as $\SL(n,\R)$ ($n \geq 3$). In light of Dani's correspondence \cite{Dani}, problems in Diophantine approximation can be studied by analysing orbits in $\SL(n,\R)/\SL(n,\Z)$. Those cases are therefore important for number theory. There are also effective results in other settings: see \cite{einsiedler-margulis-venkatesh, Margulis-Mohammadi, li-margulis} for effective equidistribution results for large closed orbits of semisimple subgroups and their applications to number theory, and \cite{lindenstrauss-margulis} for an effective density result and its application to number theory.

In this paper, we establish an effective equidistribution result for a particular type of one-parameter (non-horospherical) unipotent orbits in 
\[\SL(3, \R)/ \SL(3, \Z).\]
Recall that a point
\[
(x_1, x_2, \dots, x_d) \in \bR^d
\]
is \emph{Diophantine} if for some $\kap > 0$ we have
\[
\inf_{q \in \bN} \max_{1 \leq i \leq d} q^{d^{-1} + \kap} \langle q x_i \rangle > 0.
\]
A simple consequence of the Borel--Cantelli lemma from probability theory is that Diophantine points are typical, that is to say that the set of non-Diophantine points $(x_1, x_2, \dots, x_d) \in \bR^d$ is a set of $d$-dimensional Lebesgue measure zero. For $m \in \N$, let $C^m_b(X)$ denote the space of differentiable functions on $X$ with bounded derivatives up to order $m$. For $f \in C^m_b(X)$, let 
\[ \|f\|_{C^m_b} := \sum_{\ord (D) \le m} \|D f\|_{\infty} \]
denote the sum of supremum norms of derivatives of $f$ up to order $m$. Our new effective equidistribution result is as follows.
\begin{thm}
\label{thm:effective-ratner-main-result}
\par Let $G = \SL(3,\R)$, $\Gamma = \SL(3, \Z)$, and $X = G/\Gamma$. Let $\mu_G$ denote the $G$-invariant probability measure on $X$. Given $(a, b) \in \R^2$ such that $(a,b)$ is Diophantine, we consider the straight line
\[\{ (y, x) \in \R^2 : y = f (x) : = a x + b \}.\]
Define $u : \R^2 \to G$ by
\[ u(y, x) := \begin{bmatrix} 1 & & y \\ & 1 & x \\ & & 1 \end{bmatrix}. \]
For $t \geq 0$ and $s \geq 0$, let 
\[ a(t,s) := \begin{bmatrix} e^t & & \\ & e^s & \\ & & e^{-t-s}  \end{bmatrix} \in G. \]
Let $\cI$ be a compact subinterval of $\bR$. Then there exist constants $m \in \N$, $c \in (0,1)$, and $\eta>0$ such that for any $F \in C^{m}_b(X)$, any subinterval $J$ of $\cI$, any $s >0$, and any $0 < t \leq cs$, we have
\begin{equation}
\label{equ:effective-ratner}
\Bigl | \frac{1}{|J|}\int_J F([a(t,s) u(f(x), x)]) \dd x - \int_X F \dd \mu_G \Bigr | \ll_\cI |J|^{-1} e^{-\eta t} \|F\|_{C^m_b}.
\end{equation}
\end{thm}

\noindent Here, and throughout, we employ the Vinogradov and Bachmann--Landau notations: for
functions $f$ and positive-valued functions $g$, we write $f \ll g$ or $f = O(g)$ if
there exists a constant $C$ such that $|f| \le Cg$ pointwise, and write $f \asymp g$ if $f \ll g$ and $g \ll f$. The implied constant in \eqref{equ:effective-ratner} depends on $\cI$ but not $J$.

\begin{remark}
\par $\quad$
\begin{enumerate}
\item Let us fix a point $x_0 \in \R$. By conjugation, it is easy to verify that
\[ a(t, s) u(f(x), x) = u(r \vv h) g, \]
where $g = a(t, s) u(f(x_0), x_0)$, $r = e^{2s+t} (x - x_0)$ and $\vv h = (a e^{t-s}, 1 )$.
Therefore, the set $\{ [a(t,s)u(f(x), x)]: x \in [x_0 - R, x_0 +R ]  \}$ is equal to the one-parameter unipotent orbit
\[ \{ [u(r\vv h)g]: r \in [-R e^{2s+t}, R e^{2s+t}] \}. \]
Specialising $t = c s$ and $R = 1/2$ in Theorem \ref{thm:effective-ratner-main-result} reveals that this orbit, which has length $T = 2R e^{2s+t} = e^{(2+c)s}$, is $O(T^{-\delta})$-equidistributed in $X$, for $\delta = \eta c/(2+c)$.
\item We expect that the method can be generalised to prove effective equidistribution of one-parameter unipotent orbits in $\SL(n,\R)/\SL(n,\Z)$. This is an on-going project.
\item We expect that the method also applies when we replace the straight line by a $C^2$ planar curve, subject to a curvature assumption. This comes with additional technical difficulties, and is also work in progress. Such a result would lead to a multiplicative analogue of \cite[Theorem 1]{BDV2007}.
\end{enumerate}
\end{remark}

Compared to previous work on effective results in homogeneous dynamics, our result has the following novel attributes. First of all, our result applies to $G = \SL(3, \R)$, which is the first important case of simple Lie groups of higher rank, and the unipotent subgroup here is one-dimensional and non-horospherical. Secondly, the essence of the proof differs substantially from previous work. The main part of our proof comes from dynamical systems, rather than representation theory or Fourier analysis, although we do require Str\"ombergsson's result (see Theorem \ref{thm:strombergsson} below) on effective equidistribution for $G = \SL(2, \R) \ltimes \R^2$ which uses Fourier analysis.

\subsection{Logarithm laws in homogeneous spaces}
\label{subsec-log-laws}
\par In this subsection, we discuss an important problem in homogeneous dynamics which is closely related to Gallagher's theorem. Let us fix a non-compact homogeneous space $X = G/\Gamma$, where $G$ denotes a semisimple Lie group and $\Gamma$ denotes a non-uniform lattice in $G$, a point $o \in X$, and a subgroup
\[ A = \{ a(\vv t): \vv t \in \R^m \} \subset G \]
which is contained in a Cartan subgroup of $G$. Let us fix a right-invariant Riemannian metric $d(\cdot , \cdot)$ on $G$. Then $d(\cdot, \cdot)$ induces a metric on $X = G/\Gamma$. Let $\|\cdot\|$ denote the supremum norm on $\R^m$. Given $x \in X$, it is natural to consider the fastest rate at which the orbit 
\[
\{ a(\vv t) x : \vv t \in \R^m \}
\]
escapes to infinity as $\|\vv t\| \to \infty$, namely the asymptotic behavior of 
\[
\displaystyle \sup_{\|\vv t\| \leq R} d(a(\vv t)x, o)
\]
as $R \to \infty$.

\par This problem was first investigated by Sullivan \cite{Sullivan1982}, who considered the case in which $G = \SO(k+1,1)$ and $A = \{a(t) : t \in \R\}$ is a maximal $\R$-split Cartan subgroup of $G$, and established the following logarithm law: for almost every $x \in X$, with respect to the Haar probability measure $\mu_G$, we have
\begin{equation} \label{svsull}
 \limsup_{t \to \infty} \frac{d(a(t)x, o)}{\log|t|} = \frac{1}{k}.
 \end{equation}
Here $d(\cdot, \cdot)$ is chosen such that $\SO(k+1)\setminus\SO(k+1,1) \cong \bH^{k+1}$ is the universal $(k+1)$-dimensional hyperbolic space of sectional curvature $-1$. In this case $X$ corresponds to a non-compact, finite-volume hyperbolic manifold $ M:= \bH^{k+1}/\Gamma$ where $\Gamma$ is a geometrically finite Kleinian group of the first kind  with parabolic elements. The dynamics of $A = \{a(t): t \in \R\}$ corresponds to the geodesic flow on the unit tangent bundle of $M$. The key to the logarithm law is a Khintchine-type theorem for the action of $\Gamma$ on $\bH^{k+1}$.  This  was originally established by Patterson \cite{Patterson1976} for geometrically finite groups of the first kind and later extended to groups of the second kind in \cite{SV95}---see also \cite[\S 10.3]{BDV2006}. In view of the latter, there is a natural  analogue of \eqref{svsull} associated to any non-elementary, geometrically finite Kleinian group.

Later Kleinbock and Margulis \cite{Kleinbock-Margulis-Log-Laws} (see also \cite{Kleinbock-Margulis-Log-Laws-Erratum} for its erratum) generalised this logarithm law to a general semisimple Lie group $G$ and its diagonal subgroup $A \subset G$, that is, they showed that there exists a constant $\kappa$ depending on $G$, $\Gamma$, $d(\cdot, \cdot)$, and $A$, such that for almost every $x \in X$ we have
\[ \limsup_{\|\vv t\| \to \infty} \frac{d(a(\vv t) x, o)}{\log \|\vv t\|} = \kappa. \]
There are similar logarithm laws for unipotent flows; we refer the reader to \cite{Athreya-Margulis-I}, \cite{Athreya-Margulis-II} and \cite{Yu2017}. For a discussion of logarithm laws for hyperbolic manifolds, see \cite[\S4.2]{BGSV18}.

It is natural to consider the following finer question: given a proper submanifold $\mathcal{U}$ of $X$ not containing any open subsets of horospherical orbits, does a typical point in $\mathcal{U}$ satisfy the same logarithm law? The method of \cite{Kleinbock-Margulis-Log-Laws} relies on spectral gap properties of unitary representations of semisimple Lie groups, and thus cannot be applied to study this finer problem. In the present article, we provide a partial answer to this question for a special type of submanifold of $X = \SL(3,\R)/\SL(3,\Z)$ with the action of the diagonal semigroup
\[A^+ = \left\{a(\vv t) = \begin{bmatrix} e^{t_1} & & \\ & e^{t_2} & \\ & & e^{-t_1 - t_2}\end{bmatrix} , \vv t = (t_1, t_2) \in \R_{\ge 0}^2 \right\}. \]

For $\vv x \in \R^2$ and $\vv t = (t_1, t_2) \in \R^2$, let $u(\vv x)$ and $a(\vv t)$ be as in Theorem \ref{thm:effective-ratner-main-result}. We will see that Theorems \ref{thm:mda-result}, \ref{ConvergenceResult}, and \ref{ConvergenceVariant} imply the following.
\begin{cor}
\label{cor:log-law-line}   
For $\Lambda \in \SL(3, \R)/\SL(3,\Z)$, define 
\begin{equation} \label{DeltaDefinition}
\Delta(\Lam) = \max_{\bv \in \Lam \setminus \{\bzero\}} \log (1/ \| \bv \|).
\end{equation}

For $(a,b) \in \R^2$, let $L_{a,b}$ denote the straight line in Theorem \ref{thm:mda-result}. Let us map $L_{a,b}$ to a line $[L_{a,b}]$ in $X = \SL(3, \R)/\SL(3,\Z)$ by sending $\vv x \in \R^2$ to $[u(\vv x)] \in X$. Fix a point $o \in X$. If $\ome^*(a,b) < 5$ or $\ome^\times(a,b) < 4$ then for almost every $x \in [L_{a,b}] \subset X$ we have
\begin{equation}
\label{equ:log-law-1}
 \limsup_{\substack{\bt \in Q_1 \\ \|\vv t\| \to \infty}} \frac{\Delta(a(\vv t) x)}{\log \|\vv t\|} = \frac{2}{3},
 \end{equation}
and there exist constants $0< E_1 < E_2$ such that for almost every $x \in [L_{a,b}] \subset X$ we have
\begin{equation}
\label{equ:log-law}
 \limsup_{\substack{\bt \in Q_1 \\ \|\vv t\| \to \infty}} \frac{d(a(\vv t) x, o)}{\log \|\vv t\|} \in [E_1, E_2],
 \end{equation}
where $Q_1 = \bR_{\ge 0}^2$.
 \end{cor}
 
\begin{remark} \label{quadrant} \leavevmode
\begin{enumerate}
\item Theorem \ref{thm:mda-result} is not strictly required for the lower bound as stated; for this purpose the weaker Khintchine-based statement \eqref{Khintchine} suffices. However, the constant $E_1$ is better---that is to say greater---if one inserts Theorem \ref{thm:mda-result} as we do. We will expound upon this in Remark \ref{KhintchineRemark}.
	\item If we do not restrict $\vv t$ to lie in $Q_1$, Corollary \ref{cor:log-law-line} will no longer hold. In fact, letting $\vv t = (-t, 1)$ with $t \to +\infty$, for any $\vv x \in \R^2$ we have $d([a(\vv t) u(\vv x)], [e]) \gg \|\vv t\|$ as $\| \bt \| \to \infty$. This follows readily from the facts that $- \log \|a(\vv t) u(\vv x) \vv e_1\| = t \gg \|\vv t\|$, where $\vv e_1 = (1,0,0)$, and that if $\Delta (g \bZ^3)$ is sufficiently large then
\begin{equation} \label{DeltaFormula}
d([g],[e]) \asymp \Delta(g \bZ^3).
\end{equation}
    \item By \cite[Theorem 1.10]{Kleinbock-Margulis-Log-Laws}, the equality
    \begin{equation}
    \label{equ:log-laws-equality}
    \limsup_{ \|\vv t\| \to \infty} \frac{d(a(\vv t) x, o)}{\log \|\vv t\|} = \frac{2}{\kappa}
    \end{equation}
holds for almost every $x \in X$, without the restriction $\vv t \in Q_1$ in the $\limsup$. Here $\kappa>0$ is the unique constant for which \[\mu_G(\{x\in X: d(x, o) \geq T\}) \asymp e^{-\kappa T}\] holds whenever $T$ is sufficiently large.
 \end{enumerate}
\end{remark}

\noindent We conjecture that \eqref{equ:log-laws-equality} holds for almost every $x \in [L_{a,b}]$ if $(a,b)$ is Diophantine, if we impose the restriction $\vv t \in Q_1$ in the $\limsup$. Furthermore, we conjecture that this equality remains valid in the following general setup: given a non-compact homogeneous space $X = G/\Gamma$ where $G$ is semisimple, a diagonal subgroup $A = \{a(\vv t) : \vv t \in \R^m\}$ of $G$, and a proper submanifold $\mathcal{U}$ in $X$ satisfying a ``natural'' Diophantine condition, we have that for almost every $x \in \mathcal{U}$, the orbit $\{a(\vv t)x : \vv t \in \R^m\}$ follows the same logarithm law as a typical point in $X$.

\subsection{Organisation.}
\label{subsec-organization-paper}
\par The paper is arranged as follows. In \S \ref{sec-proof-effective-ratner}, we establish Theorem \ref{thm:effective-ratner-main-result} (effective equidistribution). In \S \ref{sec-MDA-lines}, we use Theorem \ref{thm:effective-ratner-main-result} to complete the proof of Theorem \ref{thm:mda-result} (Gallagher's theorem on planar lines). In \S \ref{ConvergenceTheory}, we prove Theorems \ref{ConvergenceResult} and \ref{ConvergenceVariant} (convergence theory). Finally, in \S \ref{sec-proof-log-law-lines}, we prove Corollary \ref{cor:log-law-line} (logarithm law).

\subsection{Funding and Acknowledgments} SC was supported by EPSRC Programme Grant EP/J018260/1 and EPSRC Fellowship Grant EP/S00226X/1. LY was supported by NSFC grant 11743006 and startup research funding from Sichuan University, and is supported by the Shiing-Shen Chern Membership at the Institute for Advanced Study.
\par We thank Demi Allen, Victor Beresnevich and Sanju Velani for valuable discussions on this topic and for helpful comments on an earlier version of this manuscript. We also thank Asaf Katz, Elon Lindenstrauss and Barak Weiss for their interest, encouragement and feedback. We are grateful to the anonymous reviewer for helpful comments. LY thanks the University of York for their hospitality during his visit when the project began.

\section{Effective equidistribution}
\label{sec-proof-effective-ratner}
\par This section is devoted to the proof of Theorem \ref{thm:effective-ratner-main-result}. Let us fix some notation before proceeding in earnest. Let
\begin{equation}
\label{equ:def-H}
H := \left\{ h(g):= \begin{bmatrix} 1 & \\ & g \end{bmatrix}: g \in \SL(2, \R) \right\}
\subset \SL(3, \R).
\end{equation}
Plainly, the subgroup $H$ is isomorphic to $\SL(2, \R)$. Put $\Gamma_H : = H \cap \Gamma$. Since $\Gamma_H \cong \SL(2, \Z)$, the orbit $[H] \subset X$ of $H$ is closed and is isomorphic to $\SL(2, \R)/\SL(2, \Z)$.
\par Let $\Aff(H) \subset \SL(3, \R)$ denote the subgroup
\begin{equation}
\label{equ:def-Aff-H}
\Aff(H) := \left\{ \begin{bmatrix} 1 & \vv v \\ & g \end{bmatrix}: \vv v \in \R^2, g \in \SL(2, \R) \right\}.
\end{equation}
Then
\[\Aff(H) =  H \ltimes \R^2 \cong  \SL(2, \R) \ltimes \R^2, \]
 where $H \cong \SL(2,\R)$ acts on $\R^2$ by right matrix multiplication. We write $\Gamma_{\Aff(H)} := \Aff(H) \cap \Gamma$, and note that
\[\Gamma_{\Aff(H)} =  \Gamma_H \ltimes \Z^2 \cong  \SL(2,\Z) \ltimes \Z^2 \]
is a lattice in $\Aff(H)$. This implies that the orbit $[\Aff(H)]$ of $\Aff(H)$ is closed and isomorphic to $ \SL(2, \R) \ltimes \R^2 / \SL(2,\Z) \ltimes \Z^2 $. In the sequel, we denote
\[ \begin{bmatrix} 1 & \vv v  \\ & g \end{bmatrix} \in \Aff(H) \]
by $(g , \vv v)$. With this notation, we have
\[ (g_1, \vv v_1) (g_2, \vv v_2) = (g_1 g_2 , \vv v_1 g_2 + \vv v_2). \]
\par Define
\[ D := \left\{ d(\lambda) := \begin{bmatrix} 1 & & \\ & e^{\lambda} & \\ & & e^{-\lambda} \end{bmatrix} \in H:  \lam  \in \R \right\} \]
and
\[ V := \left\{ v(r) := u(r \vv e_2) = \begin{bmatrix} 1 & & \\ & 1 & r \\ & & 1 \end{bmatrix} \in H : r \in \R   \right\}. \]
For $t \in \R$, write
\[ \xi(t) :=  \begin{bmatrix} e^{2t} & & \\ & e^{-t} & \\ & & e^{-t} \end{bmatrix},\]
and note that $a(t,s) = \xi(t/2) d(s+t/2)$. 
\par Let $\vv e_1 := (1,0)$ and $\vv e_2 := (0,1)$ be the standard basis vectors. Then for $x \in \R$ we have
\[
\begin{array}{rcl}
a(t,s ) u(\varphi(x)) & = & \xi(t/2) d(s + t/2)  v(x) u(f(x) \vv e_1) \\
                      & = & \xi(t/2) d(s + t/2)  v(x) (\I, f(x) \vv e_2),
\end{array} \]
where $\varphi(x) = (f(x),x)$. We also compute that
\begin{align*}
& (\I, -a e^{-s-t/2} \vv e_1) d(s + t/2)  v(x) (\I, (ax + b) \vv e_2) \\
& \qquad =  d(s+t/2) (\I, -a  \vv e_1) v(x) (\I, (ax + b) \vv e_2) \\
& \qquad = d(s+t/2) v(x) (\I, (-a, b)).
\end{align*}
Thus, the difference (with respect to the group operation of $\Aff(H)$) between $a(t,s) u(\varphi(x))$ and $\xi(t/2)d(s+t/2) v(x) (\I, (-a, b))$ is
\[ \xi(t/2) (\I, -a e^{-s-t/2} \vv e_1) \xi(-t/2) = (\I, -a e^{-s+t} \vv e_1), \]
which is exponentially close to the identity. Therefore, to prove Theorem \ref{thm:effective-ratner-main-result}, it suffices to show that
\[ \{ [\xi(t/2) d(s+t/2) v(x) (\I, (-a, b))]: x \in J \} \]
is $O(|J|^{-1}e^{-\eta t})$-equidistributed for some constant $\eta > 0$.

Note that 
\[
[d(s+t/2) v(x) (\I, (-a, b))] \in [\Aff(H)] \cong \SL(2,\R)\ltimes \R^2/\SL(2,\Z)\ltimes \Z^2,
\]
so we may apply the following result due to Str\"ombergsson \cite[Theorem~1.2]{Strombergsson2015}, which is an effective equidistribution theorem for $\SL(2,\R)\ltimes \R^2 / \SL(2, \Z) \ltimes \Z^2$.

\begin{thm}[Str\"ombergsson]
\label{thm:strombergsson}
Let $L = \SL(2, \R) \ltimes \R^2$ and $\Lambda = \SL(2, \Z) \ltimes \Z^2$. We denote an element in $L$ by $(h, \vv v)$ as above, where $h \in \SL(2, \R) $ and $\vv v \in \R^2$. For $h \in \SL(2,\R)$, let us simply denote $(h, \vv 0)$ by $h$ and treat $\SL(2,\R)$ as a subgroup of $L$.  Write
\[ d(s) := \begin{bmatrix} e^s & \\ & e^{-s} \end{bmatrix} \in \SL(2, \R) \subset L \]
and
\[ u(r) := \begin{bmatrix} 1 & r \\ 0 & 1 \end{bmatrix} \in \SL(2, \R) \subset L. \]
Let $\cI$ be a fixed compact subinterval of $\bR$. Then for any $\epsilon >0$, any subinterval $J$ of $\cI$, and any $\vv v =(v_1, v_2)$, the orbit
\begin{equation} \label{StromOrbit}
\{ [d(s)u(r)(\I, \vv v )] : r \in J \}
\end{equation}
is $O(|J|^{-1} (e^{-s/2} + b_s(\vv v))^{1-\epsilon})$-equidistributed in $L/\Lambda$. That is, there exists a constant $C = C(\eps, \cI) > 0$ such that for any $F \in C^{8}_b(L/\Lambda)$ we have
\[ \left| \frac{1}{|J|} \int_{J} F([d(s)u(r)(\I , \vv v) ]) \dd r - \int F \dd \mu_L \right|  \leq C\|F\|_{C^8_b}|J|^{-1} (e^{-s/2} + b_s(\vv v))^{1-\epsilon}, \]
where
\[b_s (\vv v) = \max_{q \in \bN} \: \min \left\{ \frac{1}{q^2}, \frac{e^{-s/2}}{q\|q v_1\|},\frac{e^{-s/2}}{q\|q v_2\| }  \right\}.\]
In particular, if $\vv v $ is Diophantine, so that there exists $\kap >0$ for which
\[ \inf_{q \in \bN}  q^{1+\kap} ( \|q v_1\| + \|q v_2\|)  > 0,\]
then the orbit \eqref{StromOrbit} is $O_{\eps, \cI}(|J|^{-1}e^{-\eta(1-\epsilon) s})$-equidistributed, where $\eta = (\kappa + 2)^{-1}$.
\end{thm}

\par In addition to Str\"ombergsson's effective equidistribution theorem, we require the following result.
\begin{thm}
\label{thm:effective-horospherical-subgroup}
There exists a constant $\eta_1 >0$ such that for any $t >0 $, the orbit
\[ \{ [\xi(t/2) (\I, \vv v) ] : \vv v \in [0,1]^2  \} \]
is $O(e^{-\eta_1 t})$-equidistributed in $X$, that is, there exists $m \in \N$ such that for any $F \in C^m_b(X)$, 
\[ \left| \int_{[0,1]^2} F([\xi(t/2) (\I, \vv v) ]) \dd \vv v - \int_X F \dd \mu_G \right| \ll \|F\|_{C^m_b} e^{-\eta_1 t}. \]
\end{thm}

\noindent This theorem can be proved using the standard ``thickening'' method developed in Margulis's thesis \cite{Margulis2004}, noting that $\{(\I, \vv v): \vv v \in \R^2\}$ is the expanding horospherical subgroup of $\xi(t)$, and that $\SL(3, \R)$ has Kazhdan's property (T). The reader is referred to \cite[Theorem 1.3]{Klein-Mar-Effective-Equid} for a proof.

\par We are now equipped to prove Theorem \ref{thm:effective-ratner-main-result}.

\begin{proof}[Proof of Theorem \ref{thm:effective-ratner-main-result}]

By the preceding discussion, it remains to show that
\[ \{ [\xi(t/2) d(s+t/2) v(x) (\I, (-a, b))]: x \in J \} \]
is $O(|J|^{-1}e^{-\eta t})$-equidistributed for some constant $\eta > 0$. 

We begin by considering $d(s + t/2)  v(x) (\I, (-a,b))$. Since $(a,b)$ being Diophantine implies that $(-a, b)$ is Diophantine, we conclude from Theorem \ref{thm:strombergsson} that the orbit
 \[ \{ [ d(s + t/2)  v(x) (\I, (-a,b))]: x \in J \} \]
 is $O(|J|^{-1} e^{-\tilde{\eta} s})$-equidistributed in $[\Aff(H)]$, for some constant $\tilde{\eta} >0$.

We may assume that $\eta < \tilde{\eta}$. As the desired conclusion is trivial when $|J| < e^{-\eta t}$, we may also assume that $|J| \geq e^{-\eta t} \geq e^{-\tilde{\eta} t}$. We shall choose $c < 1/2$ in Theorem \ref{thm:effective-ratner-main-result}. Now
 \[ |J|^{-1} e^{-\tilde{\eta} s} \leq e^{\tilde{\eta} t - \tilde{\eta} s } \leq e^{\tilde{\eta} c s - \tilde{\eta} s} \leq e^{-\tilde{\eta} s /2}, \]
and so the orbit 
 \[ \{ [ d(s + t/2)  v(x) (\I, (-a,b))]: x \in J \} \]
 is $O(e^{-\tilde{\eta} s/2})$-equidistributed in $[\Aff(H)]$.

 \par Let us fix a fundamental domain $\mathcal{F}_0 \subset \SL(2, \R)$ for $\SL(2, \R)/\SL(2, \Z)$. For any $\epsilon >0$, define
\[
K_\eps = \{ \Lam \in \SL(2, \R)/\SL(2, \Z): \Lam \cap B_\eps(\bzero) = \{ \bzero \} \}
\]
and
\[ \mathcal{K}_{\epsilon} := \{ g \in \mathcal{F}_0: [g] \in K_{\epsilon}\}.\]
For constants $\eta_3 > 2 \eta_2 >0$ to be determined later, we divide $\mathcal{K}_{e^{-\eta_2 t}}$ into small pieces of radius $e^{-\eta_3 t}$:
\[ \{ \mathcal{P}_i: i=1, \dots, \cL \}. \]
Without loss of generality, we may assume that every $\mathcal{P}_i$ has the same measure, considering the Haar measure $\mu_H$ on $H \cong \SL(2, \R)$. Note from \cite[Proposition 3.5]{Klein-Mar-Effective-Equid} that the injectivity radius of $\cK_{\eps}$ is $r(\cK_\eps) \gg \eps^2$, so the fact that $\eta_3 > 2 \eta_2$ enables us to perform this subdivision. Since $\cP_i$ is three-dimensional, we now have
\begin{equation}
\label{equ:size-cP_i}
\mu_H(\cP_i) \asymp e^{-3\eta_3 t}.
\end{equation}

Let 
\[ \pi_1 :  \SL(2, \R)\ltimes \R^2 / \SL(2, \Z) \ltimes \Z^2 \to \SL(2,\R)/\SL(2,\Z)\]
denote the projection mapping that sends $[(g, \vv v)] \in \SL(2, \R)\ltimes \R^2 / \SL(2, \Z) \ltimes \Z^2$ to its first component $[g] \in \SL(2,\R)/\SL(2,\Z)$. For each $i$, let
\[ J_i = \{ x \in J:  \pi_1([ d(s + t/2)  v(x) (\I, (-a,b))]) \in [\mathcal{P}_i]\}. \]
By Theorem \ref{thm:strombergsson}, there exists a constant $\tilde{\eta}' >0$ such that
\begin{equation} 
\label{equ:size_J_i}
|J_i| / |J| = \mu_H(\mathcal{P}_i) +O(e^{-\tilde{\eta}' s}). 
\end{equation}
Indeed, this can be formally established by approximating $1_{[\cP_i]}\circ \pi_1 $ by a smooth function $\tilde{1}_{[\cP_i]} \in C^8_b([\Aff(H)])$ with $O(e^{-\tilde{\eta}' s})$-error and applying Theorem \ref{thm:strombergsson} to $\tilde{1}_{[\cP_i]}$. Since the smoothing is standard---see \cite[\S 3]{Kleinbock-Margulis-Log-Laws-Erratum} for instance---we omit the details.
\par We shall choose $c>0$ small enough so that $3\eta_3 c < \tilde{\eta}'/2 $. This ensures that
\[ e^{-3\eta_3 t} \geq e^{-3\eta_3 c s} > e^{-\tilde{\eta}' s/2}. \]
 Then by \eqref{equ:size-cP_i}, we have
 \begin{equation}
 \label{equ:size_J_i_2}
 |J_i|/|J| \asymp e^{-3\eta_3 t}.
 \end{equation}
By \cite[Proposition 7.1]{Kleinbock-Margulis-Log-Laws}, we have
\[
\mu_H(\mathcal{F}_0 \setminus \mathcal{K}_{e^{-\eta_2 t}} ) = O(e^{-2\eta_2 t}).
\]
Thus, in order to show that
\[ \{[ \xi(t/2) d(s + t/2)  v(x) (\I, (-a,b))]: x \in J\} \]
is $O(|J|^{-1}e^{-\eta t})$-equidistributed, it suffices to show that for each $i$, the orbit
\[ \{[ \xi(t/2) d(s + t/2)  v(x) (\I, (-a,b))]: x \in J_i \}\]
is $O(|J|^{-1}e^{-\eta t})$-equidistributed. 
\par We now focus our attention on some $J_i$, and forge ahead with our analysis of $d(s+t/2)v(x)(I,(-a,b))$. Let
\[\pi_2 : \SL(2, \R)\ltimes \R^2 / \SL(2, \Z) \ltimes \Z^2 \to \R^2/\Z^2\]
denote the projection mapping that sends $[(g, \vv v)] \in \SL(2, \R)\ltimes \R^2 / \SL(2, \Z) \ltimes \Z^2$ (where $g \in \cF_0$) to its second component $[\vv v] \in \R^2/\Z^2$. By Theorem \ref{thm:strombergsson}, we have that the second component of
\[ \{ [ d(s + t/2)  v(x) (\I, (-a,b))] : x \in J_i \}\]
is $O(|J_i|^{-1}e^{-\eta' s})$-equidistributed in $\R^2/\Z^2$ for some constant $\eta' >0$, that is, for any $f \in C^8_b(\R^2/\Z^2)$ we have
\begin{equation*}
 \left| \frac{1}{|J_i|} \int_{J_i} f(\pi_2 ([g(x)])) \dd x - \int_{\R^2/\Z^2} f([\vv v]) \dd \vv v \right| \ll |J_i|^{-1} \|f\|_{C^8_b} e^{-\eta' s }, \end{equation*}
where $g(x) = d(s + t/2)  v(x) (\I, (-a,b))$. Similarly to \eqref{equ:size_J_i}, this can be formally established by approximating $(1_{[\cP_i]}\circ \pi_1 ) (f \circ \pi_2)$ by a smooth function $\tilde{f}_i \in C^8_b([\Aff(H)])$ with $O(e^{-\eta' s})$-error and applying Theorem \ref{thm:strombergsson} to $\tilde{f}_i$. 

Next, let us fix some $g_i \in \mathcal{P}_i \subset G$. For any $x \in J_i$, the first component of
\[ [ d(s + t/2)  v(x) (\I, (-a,b))] \]
can be written as $O_\fe(e^{-\eta_3 t}) g_i$, where $\fe \in \mathrm{SL}(3,\bR)$ denotes the identity. Here, and in the calculation below, we write $O_\fe(\mathfrak{r})$ for an element of a neighbourhood of $\fe$ whose radius is $O(\mathfrak{r})$. Note that $\xi(t/2)$ commutes with $(g, \vv 0)$, so for $x \in J_i$ we have
\[
\begin{array}{cl}
  [ \xi(t/2) d(s + t/2)  v(x) (\I, (-a,b))]
& = [\xi(t/2)(O_\fe(e^{-\eta_3 t})g_i, \vv v(x) )] \\
 &=  [\xi(t/2) ( O_\fe(e^{-\eta_3 t})g_i, \vv 0 ) (\I , \vv v(x)) ]\\
 &= [ ( O_\fe(e^{-\eta_3 t})g_i, \vv 0 ) \xi(t/2) (\I , \vv v(x)) ] ,
\end{array}\]
where $\vv v(x)$ denotes the second component of
\[
[ d(s + t/2)  v(x) (\I, (-a,b))] .
\]
It therefore remains to show that
\[ \{ [ (g_i , \vv 0) \xi(t/2)(\I, \vv v(x)) ] : x \in J_i \} \]
is $O(|J|^{-1}e^{-\eta t})$-equidistributed.\\

Since $g_i \in \mathcal{K}_{e^{-\eta_2 t}}$, there exists a constant $N_1 >1$ such that $\|g_i\| \ll e^{N_1 \eta_2 t}$. Let $m$ be the positive integer given in Theorem \ref{thm:effective-horospherical-subgroup}. By replacing $m$ with a larger integer if needed, we may assume that $m \geq 8$. Given $F \in C^m_b(X)$ such that $\|F\|_{C^m_b} \leq 1$, we need to show that
\begin{equation}  
\label{equ:goal-estimate}
\left|\frac{1}{|J_i|} \int_{J_i} F([ (g_i , \vv 0) \xi(t/2)(\I, \vv v(x)) ]) \dd x - \int_X F \dd \mu_G \right| \ll |J|^{-1}e^{-\eta t} \end{equation}
holds for some constant $\eta >0$. The triangle inequality gives
\[
 \left|\frac{1}{|J_i|} \int_{J_i} F([ (g_i , \vv 0) \xi(t/2)(\I, \vv v(x)) ]) \dd x - \int_X F \dd \mu_G \right| \le X_1 + X_2,
\]
where
\[
X_1 = \left| \frac{1}{|J_i|} \int_{J_i} F([ (g_i , \vv 0) \xi(t/2)(\I, \vv v(x)) ]) \dd x - \int_{[0,1]^2}F([ (g_i , \vv 0) \xi(t/2)(\I, \vv v) ]) \d \bv \right| \]
and
\[
X_2 = \left| \int_{[0,1]^2}F([ (g_i , \vv 0) \xi(t/2)(\I, \vv v) ]) \d \bv - \int_X F \dd \mu_G \right|.
\]
\par We begin by estimating $X_1$. For $g \in G$, we define $F_g \in C^m_b(X)$ by
\[
F_g (x) = F(g x),
\]
and note that
\[\|F_g\|_{C^m_b} \ll \|g\|^{m} \|F\|_{C^m_b} \le \|g\|^m.\]
We have
\[
X_1 = \left| \frac{1}{|J_i|} \int_{J_ i} F_{h_1} ([(\I, \vv v(x))]) \dd x - \int_{[0,1]^2} F_{h_1}([(\I, \vv v)])  \d \bv \right|,
\]
where $h_1 = (g_i, \vv 0) \xi(t/2)$. We choose $\eta_2 >0$ small enough so that $N_1 \eta_2 < 1$; this ensures that $\|h_1\| \ll e^{2t}$. As
\[ \{ \vv v(x) : x \in J_i\} \]
is $O(|J_i|^{-1}e^{-\eta' s})$-equidistributed in $\R^2/\Z^2$, we obtain
\[
X_1 \ll |J_i|^{-1}  e^{-\eta' s} \|F_{h_1}\|_{C^m_b}  \ll |J|^{-1} e^{3\eta_3 t} e^{-\eta' s} \|h_1\|^m
 = |J|^{-1} e^{(3\eta_3 + 2m) t  -\eta' s}.
\]
We choose $c >0$ small enough such that $(3\eta_3 + 2m) c < \eta'/2$. Now, for $t \le cs$, we have
 \begin{equation}
 \label{equ:estimate-first-term}
 X_1 \ll |J|^{-1} e^{-\eta_4 s} \leq |J|^{-1} e^{-c^{-1}\eta_4 t} , \end{equation}
 where $\eta_4 = \eta'/2$.

It remains to estimate $X_2$. Since $\mu_G$ is $G$-invariant, we have
 \[\int_X F \dd \mu_G = \int_X F_g \dd \mu_G\]
 for any $g \in G$. Therefore
 \[
X_2 = \left| \int_{[0,1]^2}F_{h_2}([\xi(t/2)(\I, \vv v) ])   \d \bv - \int_X F_{h_2} \dd \mu_G \right|,
\]
 where $h_2 = (g_i, \vv 0)$. As $\|h_2\| \ll e^{N_1 \eta_2 t}$, Theorem \ref{thm:effective-horospherical-subgroup} now gives
 \[
 X_2
    \ll e^{-\eta_1 t} \|F_{h_2}\|_{C^m_b} \ll e^{-\eta_1 t} \|h_2\|^m \ll e^{- \eta_1 t + mN_1 \eta_2 t}.
\]
 We may choose $\eta_2 >0$ small enough such that $ m N_1 \eta_2 <  \eta_1/2$, and so
 \begin{equation}
 \label{equ:estimate-second-term}
X_2 \ll e^{-\eta_5 t} \ll |J|^{-1} e^{-\eta_5 t},
 \end{equation}
 where $\eta_5 = \eta_1/2$.
 \par Combining \eqref{equ:estimate-first-term} and \eqref{equ:estimate-second-term}, we obtain \eqref{equ:goal-estimate} with $\eta = \min\{ c^{-1}\eta_4, \eta_5 \}$, and thus complete the proof.
\end{proof}


 \section{Multiplicative Diophantine approximation on planar lines}
 \label{sec-MDA-lines}

In this section, we complete the proof of Theorem \ref{thm:mda-result} using Theorem~\ref{thm:effective-ratner-main-result} and techniques from homogeneous dynamics.

Let $\cL = L_{a,b}$, as in \eqref{Lab}. We begin by dealing with the non-Diophantine case. This is a routine consequence of Kleinbock's work \cite{Klein2003} on extremal subspaces, from which we presently recall some standard definitions. A pair $(\alp, \bet) \in \bR^2$ is \emph{very well multiplicatively approximable} (VWMA) if for some $\eps > 0$ there exist infinitely many $n \in \bN$ for which
\[
\langle n \alp \rangle  \langle n \bet \rangle \le n^{-1-\eps}.
\]
(We have actually taken an equivalent definition from \cite{Kleinbock-Margulis-Log-Laws}---see the introduction of that paper for a discussion of this equivalence.) The line $\cL$ is \emph{strongly extremal} if almost all points in $\cL$ are not VWMA.

If $(a,b)$ is not Diophantine then $\ome(a,b) = \infty$. By \cite[Corollary 5.7]{Klein2003}, the line $\cL$ is then not strongly extremal so, by Theorem 5.5 therein, \textbf{all} points in $\cL$ are VWMA. The upshot is that, in this non-Diophantine case, the approximation rate \eqref{TwoLogs} is valid for all $(\alp, \bet) \in \cL$.

Having dealt with the non-Diophantine case of Theorem \ref{thm:mda-result}, we assume henceforth that $(a,b)$ is Diophantine.

\subsection{Overview}
\label{overview}

We hope that this section will serve as a general framework for deducing divergence statements in metric Diophantine approximation using effective equidistribution theorems. There are six core steps in our proof.

\begin{enumerate}

\item \textbf{Dyadic pigeonholing and a homogeneous space.} Let 
\[
f(x) = ax+b
\]
be our Diophantine linear function. We dyadically pigeonhole at levels $t,s \in \bN$:
\[
 \langle n f(x) \rangle \asymp e^{-t}, \qquad \langle nx \rangle \asymp e^{-s}.
\]
Our multiplicatively well-approximable points at these dyadic levels occur when the lattice
\[
a(t,s) u(f(x),x) \bZ^3
\]
contains a non-zero vector whose norm is at most $\eps (s+t)^{-2/3}$, where $a(t,s)$ is a particular diagonal matrix and $u(f(x),x)$ is a particular unipotent matrix. This enables us to work in the homogeneous space $\SL_3(\R)/ \SL_3(\Z)$ of unimodular lattices in $\bR^3$. We only consider when $s \asymp t$; this provides us with sufficiently many good approximations.

\item \textbf{Local divergence Borel--Cantelli.} Collecting together the well-approximable points $(f(x),x)$ at scale $(t,s)$, we obtain a limit superior set of a collection of sets $B_\eps(t,s)$. To show it has full measure, it suffices to show that its ``local'' measures are positive; these are induced probability measures on subintervals $J$. We apply divergence Borel--Cantelli for this purpose. We ``prune'' to $B^*(t,s) \subseteq B_\varepsilon(t,s) \cap J$ being the union of separated subintervals of a suitable length. Our task is now to establish quasi-independence on average for the sets $B^*(t,s)$, where $t+s$ exceeds a threshold $T_1(J)$.

\item \textbf{The non-critical case.} Here $2s+t$ is not close to $2s' + t'$. In this case it suffices to simply choose $(t,s)$ before choosing $(t',s')$, and we obtain
\[
|B^*(t,s) \cap B^*(t',s')| \ll |J|^{-1} |B^*(t,s)| \cdot |B^*(t',s')|.
\]
\item \textbf{The critical case, a product formula, and smoothing.} Here $2s+t$ is close to $2s' + t'$. Since $s \asymp t$, we are able to infer that $s \asymp t \asymp s' \asymp t'$. Considering the two lattices, the distinction is left-multiplication by the matrix $g = a(t'-t, s'-s)$, and $\| g \| \le e^{2(c_2-c_1)s}$ is not too large. We smoothly approximate the indicator function of $B^*(t,s) \cap B^*(t',s')$ by $F$, where $F(x) = F_\ell (x) F_{\ell'}(gx)$ with $\ell = s+t$, $\ell' = s'+t'$, and $F_\ell$ being a smoothing of  the indicator function of some subset (defined by the parameter $\ell$) near the cusp. The smoothing ensures that $F$ has small complete bounded and Sobolev norms.

\item \textbf{Effective equidistribution.} By our principal result, Theorem \ref{thm:effective-ratner-main-result}, the mean of $F$ over $J$ is roughly the mean of $F$ over the entire homogeneous space. The error is exponentially-decaying in $s$ and requires control of a complete bounded norm.

\item \textbf{Exponential mixing.} By work of Kleinbock and Margulis \cite{Kleinbock-Margulis-Log-Laws}, the mean of $F$ is roughly the mean of $F_\ell$ times the mean of $F_{\ell'}$. The error is exponentially-decaying in $\max\{|s-s'|, |t-t'|\}$ and requires control of a Sobolev norm. The latter two means are as expected, owing to the careful smoothing, and we obtain 
\[
|B^*(t,s) \cap B^*(t',s')| \approx |J|^{-1} |B^*(t,s)| \cdot |B^*(t',s')|,
\]
 up to a constant multiplicative error and an exponentially-small additive error.

\end{enumerate}

Step (1) is a completely classical passage; Steps (2) and (3) are very much in the spirit of Beresnevich--Haynes--Velani~\cite{Beresnevich-Haynes-Velani2015} and the preceeding work on measure-theoretic laws for limsup sets~\cite{BDV2006}; and Steps (4), (6) are standard after Kleinbock--Margulis~\cite{kleinbock1996bounded, Kleinbock-Margulis-Log-Laws, Kleinbock-Margulis-Log-Laws-Erratum}. The crucial ingredient, used in Step (5), is our new effective equidistribution theorem. In the ensuing two subsections, we carry out the strategy by supplying concrete details.


 \subsection{Diophantine approximation to homogeneous dynamics}
 \label{subsec-DA-to-HD}
 \par The purpose of this subsection is to explain how to translate the problem in multiplicative Diophantine approximation to a problem in homogeneous dynamics.
 \par Let $G = \SL_3(\R)$ and $\Gamma = \SL_3(\Z)$. Then the homogeneous space $X = G/\Gamma$ parametrises the space of unimodular lattices in $\R^3$, where $[g]$ corresponds to the lattice $g \Z^3$. For $\epsilon >0$, let
 $B_{\epsilon}(\vv 0)$ denote the closed ball of radius $\epsilon$ and centred at $\vv 0$, with respect to the supremum norm. Let us define
 \[ K_{\epsilon} := \{ \Lambda \in X : \Lambda \cap B_{\epsilon} (\vv 0) = \{ \vv 0\} \}.\]
 Mahler's criterion asserts that $K_{\epsilon}$ is compact, and that every compact subset of $X$ is contained in some $K_{\epsilon}$. For $\vv x = (x_1, x_2 )\in \R^2$, define
 \[ u(\vv x) = \begin{bmatrix} 1 & & x_1 \\ & 1 & x_2 \\ & & 1 \end{bmatrix} \in G. \]
 For $s\geq 0$ and $t \geq 0$, define
 \[ a(t,s) = \begin{bmatrix} e^{t} & & \\ & e^s & \\ & & e^{-s-t} \end{bmatrix} \in G. \]

Let us fix a Diophantine vector $(a,b) \in \bR^2$, and write $f(x) := ax + b$. Then $L_{a,b}$ is given by
 \[ \{ \varphi(x) = (f(x), x): x \in \R \}. \]
It suffices to consider a compact segment $\{ \varphi(x): x \in \cI \}$ of $L_{a,b}$, where $\cI$ is an arbitrary fixed compact interval in $\bR$. For $t >0$, $s >0$ and $\eps \in (0,1)$, let
 \[ B_{\epsilon}(t,s) = \{ x \in \cI : [a(t,s) u(\varphi(x))] \in X\setminus K_{\epsilon \ell^{-2/3}}\}, \]
 where $\ell = s+t$.
 \par By definition, for any $x \in B_{\epsilon}(t,s)$ there exists $\vv v = (p_1, p_2, n) \in \Z^3\setminus\{\vv 0\}$ such that
 \[ \| a(t,s)u(\varphi(x))\vv v\| \leq \epsilon\ell^{-2/3}, \]
where $\|\cdot\|$ denotes the supremum norm. Therefore
\begin{equation} \label{FirstTwo}
|e^t (f(x)n + p_1 )| \leq \epsilon \ell^{-2/3}, \qquad |e^s ( x n + p_2)| \leq \epsilon \ell^{-2/3}
\end{equation}
 and
 \[ e^{-t-s} |n| \leq \epsilon \ell^{-2/3}. \]
 Without loss of generality, we may assume that $n >0$. From \eqref{FirstTwo}, we see that $|f(x) n + p_1| = \langle f(x) n\rangle$ and $|x n + p_2| = \langle x n\rangle$. Therefore
 \[ 0 <  n \leq e^{t+s} \epsilon \ell^{-2/3},  \qquad \langle n x\rangle \leq \epsilon \ell^{-2/3} e^{-s}
\]
and
\[
\langle n f(x)\rangle \leq \epsilon \ell^{-2/3} e^{-t},
\]
all of which implies that
\[ n \langle n x\rangle  \langle n f(x)\rangle \leq  \epsilon^3 \ell^{-2} \leq \epsilon^3 (\log n)^{-2}.\]
Let us take countably many $(t_k , s_k)$ such that $t_k + s_k \to +\infty$. Then, for any $\displaystyle x \in \limsup_{k \to \infty} B_{\epsilon}(t_k, s_k)$, we have
\[ \liminf_{n \to \infty} n (\log n)^2 \langle n x \rangle  \langle n f(x)\rangle \leq \epsilon^3. \]
Next, we choose a sequence $\epsilon_j \to 0$. For any $\displaystyle x \in \bigcap_{j=1}^{\infty} \limsup_{k \to \infty} B_{\epsilon_j}(t_k, s_k)$, we have
\[ \liminf_{n \to \infty} n (\log n)^2 \langle n x\rangle  \langle n f(x)\rangle = 0. \]
Thus, in order to prove Theorem \ref{thm:mda-result}, it suffices to show that if $\epsilon >0$ then $\displaystyle \limsup_{k \to \infty} B_{\epsilon}(t_k, s_k)$ has full measure. We carry this out in the next section, using the divergence Borel--Cantelli lemma.
\par Let $0 < c_1 < c_2$ be two constants which will be determined later, such that $c_2 \leq c$, where $c \in (0,1)$ is the constant that we get from Theorem \ref{thm:effective-ratner-main-result}. We will choose $\mathcal{R} = \{(t_k, s_k): k \in \Z^+\}$ as follows:
\begin{equation}
\label{equ:net-on-diagonal-group}
\mathcal{R} : = \{ (t, s): t , s \in \Z^+ :  c_1 s \leq t \leq c_2 s \}.
\end{equation}
 Since $\mathcal{R}$ is countable, we can order it as $\{(t_k, s_k): k \in \Z^+\}$.
\par We henceforth fix $\eps \in (0,1)$. In summary, to establish Theorem \ref{thm:mda-result} it suffices to prove the following.
\begin{prop}
\label{prop:limsup}
\par Let $B_{\epsilon}(t, s)$ and $\mathcal{R} = \{(t_k, s_k): k \in \Z^+\}$ be as above. Then, for any $\epsilon >0$, we have
\[ |\limsup_{k \to \infty} B_{\epsilon}(t_k, s_k)| =  |\cI| . \]
\end{prop}


\subsection{Divergent part of Borel--Cantelli lemma}
\label{subsec-Borel-Cantelli-lemma}

We will use the divergent part of the Borel--Cantelli lemma to prove Theorem \ref{thm:mda-result}. The version stated below is from  \cite[Proposition 2]{BDV2006}.

\begin{lemma}
\label{lm:borel-cantelli-divergent}
\par Let $(\Omega, A, \nu)$ be a probability space, and let $(E_n)_{n=1}^\infty$ be a sequence of measurable sets such that $\displaystyle \sum_{n=1}^\infty \nu(E_n) = \infty$. Suppose there exists a constant $C >0$ such that
\begin{equation}
\sum_{n, m =1}^N \nu(E_n \cap E_m) \leq C \left( \sum_{n=1}^N \nu(E_n) \right)^2
\end{equation}
holds for infinitely many $N \in \bN$. Then
\[ \nu(\limsup_{n \to \infty} E_n) \geq \frac{1}{C}. \]
\end{lemma}

We also require the following special case of \cite[Proposition 1]{BDV2006}. This is a well-known consequence of the Lebesgue density theorem.

\begin{lemma}
\label{lm:lebesgue-density}
Let $\cI$ be a fixed compact subinterval in $\bR$, and let $E$ be a Borel subset of $\cI$. Assume that there exists $C >0$ such that for any subinterval $J \subset \cI$ we have
\[ |E \cap J| \geq \frac{1}{C} |J|. \]
Then $|E| = |\cI|$.
\end{lemma}

For the remainder of this section, we fix a subinterval $J \subset \cI$. We'll apply Lemma \ref{lm:borel-cantelli-divergent} with $A = J$ and $\nu(E) = |J|^{-1} |E|$ for $E \subseteq J$. For each $(t,s) \in \mathcal{R}$, we will carefully choose a subset $B^{\ast}(t,s)$ of $B_{\epsilon}(t,s)\cap J$. By Lemmas \ref{lm:borel-cantelli-divergent} and \ref{lm:lebesgue-density}, in order to prove Proposition \ref{prop:limsup} it suffices to show that
\begin{equation}
\label{equ:borel-cantelli-divergent}
\sum_{(t,s) \in \mathcal{R}} |B^{\ast}(t,s)| = \infty,
\end{equation}
and that if $N$ is sufficiently large in terms of $J$ then
\begin{equation}
\label{equ:borel-cantelli-independent}
\displaystyle \sum_{(t,s), (t',s') \in \mathcal{R}\cap V(N)} |B^{\ast}(t,s) \cap B^{\ast}(t',s')|
\ll  |J|^{-1} \left( \sum_{(t,s) \in \mathcal{R}\cap V(N)}  |B^{\ast}(t,s)| \right)^2,
\end{equation}
where $V(N) = \{ (t,s) \in \bN^2: t+s \leq N\}$, and wherein the implied constant does not depend on $J$. 
Here we work with $B^{\ast}(t,s)$ instead of $B_{\epsilon}(t,s)\cap J$ to simplify the proof of \eqref{equ:borel-cantelli-independent}; this idea was also used in \cite[\S 10]{Beresnevich-Haynes-Velani2015}.

\par We will make frequent use of the calculation
\begin{equation}
\label{equ:conjugation}
a(t,s)u(\varphi(x_0 + r e^{-2s-t})) = u(O_f(e^{t-s})) u(r \vv e_2) a(t,s) u(\varphi(x_0)),
\end{equation}
valid for $r \in [-3,3]$, which is straightforward to verify by hand. We define $B^{\ast}(t,s)$ as follows.

\begin{defn}
\label{def:subset-B-ast-t-s}
For $(t,s) \in \mathcal{R}$, let us consider the pair $(t^*, s^*)$ given by
\[
t^* = t + \kap  \text{ and } s^* = s - 2 \kap,
\]
where $\kap = \lfloor\frac{2}{3} \log (s+t) - \log \epsilon \rfloor +1$. Let $T_1 = T_1(J)$ be sufficiently large. 
Let $c >0$ be a constant which will be determined later. 
For $(s,t) \in \mathcal{R}$ such that $s+t < T_1$, let us define $B^{\ast}(t,s) = \emptyset$. When $s+t \ge T_1$, we divide $J$ into small subintervals of length $2 \times 10^6 c^{-1}  e^{-2s^* - t^*}$. 
For each such subinterval $I = [x_l , x_r]$.  
For any $x \in I$, if
\[ \min \{ |x - x_l|, |x - x_r| \} \ge 4 e^{-2 s^\ast - t^\ast},\]
and there exists $\vv a \in \Z^3 \setminus \{\vv 0\}$ such that
\[ a(t^*,s^*)u(\varphi(x_0))\vv a =: \vv v = (v_1, v_2, v_3) \]
satisfies $\|\vv v\| \le 1$ and $|v_3| \ge 1/2$, we will call $x$ a good point. Otherwise we call it a bad point. 
If $I = [x_l, x_r]$ contains a good point, we call it a good interval. Otherwise we call it a bad interval.
For every good interval $I$, let us take a good point $x \in I$. Let us pick $\vv a \in \Z^3 \setminus \{\vv 0\}$ such that 

\[ \vv v = (v_1, v_2, v_3) =a(t^*, s^*) u(\varphi(x))\vv a \]
satisfies $\|\vv v\| \le 1$ and $|v_3| \ge 1/2$. For $x' = x + r e^{-2 s^\ast - t^\ast}$, \eqref{equ:conjugation} yields
\begin{equation} 
  \label{BoundedShift}
 a(t^*, s^*) u(\varphi(x'))\vv a =  u(O(e^{t^{*}-s^{*}})) u(r \vv e_2) \vv v = (v_1 , v_2 + r v_3, v_3)^{T} + O(e^{t^{*} - s^{*}}).
\end{equation}

Then there exists a unique $r^{*} \in  [-2, 2]$ such that $v_2 + r^{*} v_3 = 0$. Let us denote
\[
x_1 := x + r^{*} e^{-2s^* - t^*}, \qquad I_1 := [x_1 - \theta e^{-2s^* - t^*}, x_1+ \theta e^{-2s^* - t^*}] ,
\]
where $\theta = e^{-3\kappa}$. From \eqref{BoundedShift}, we see that if $x' \in I_1$ then
\[ a(t^*, s^*)u(\varphi(x'))\vv a =: \vv v(x') = (v_1(x'), v_2(x'), v_3(x'))^{T} \]
satisfies that $|v_2(x')| \leq \theta$. Therefore
\[
a(t,s) u(\varphi(x')) \vv a = a(- \kap, 2 \kap) \vv v(x') = (e^{- \kap} v_1(x'), e^{2 \kap} v_2(x'), e^{- \kap} v_3(x') )
\]
for all $x' \in I_1$. The supremum norm of this vector is less than or equal to $e^{- \kap} 
\le \epsilon (s+t)^{-2/3}$, so
\[ I_1 \subset B_{\epsilon}(t,s).\]
For every good interval $I$, we define $I_1 \subset I$ as above. We define $B^{\ast}(t,s)$ to be the union of the intervals $I_1$ constructed as above.
\end{defn}

\begin{proof}[Proof of Proposition \ref{prop:limsup}]
Let $\mathcal{M} \subset X$ denote the set of lattices 
containing a vector $\vv v  = (v_1 , v_2, v_3)$ with $\|\vv v\| \le 1$ and $|v_3| \ge 1/2$. It is easy to 
see that we there exists $F \in C_c^\infty (X)$ with $0 \le F \le 1$, $\|F\|_{C^m_b} \ll 1$, $\int F \dd \mu_G >0$, and $\supp F \subset \cM$. 
Now let the constant $c$ in Definition \ref{def:subset-B-ast-t-s} be $\int F \dd \mu_G/2$. 
Given $t$ and $s$, let us denote $s^\ast$ and $t^\ast$ as in Definition \ref{def:subset-B-ast-t-s}.
Applying Theorem \ref{thm:effective-ratner-main-result} with $F$ and $J$, we have 
\[ \left| \int_J F([a(t^\ast,s^\ast)u(\varphi(x))]) \dd x - |J| \int F \dd \mu_G \right| \ll e^{-\eta t}.  \]
For $t + s$ large enough, we have  
\[ \int_J F([a(t^\ast,s^\ast)u(\varphi(x))]) \dd x \ge |J| \int F \dd \mu_G /2 =  c |J|. \]
Note that 
\[ \int_J F([a(t^\ast,s^\ast)u(\varphi(x))]) \dd x \le \left| \{ x \in J: [a(t^\ast,s^\ast)u(\varphi(x))] \in \cM \}  \right|,  \] 
we get 
\[\left| \{ x \in J: [a(t^\ast,s^\ast)u(\varphi(x))] \in \cM \}  \right| \ge c |J|.\] 
Let us denote 
\[J_{good} : = \bigcup_{I \text{ good}} I,\]
and
\[  J_{bad} := \bigcup_{I \text{ bad}} I.  \]
Note that 
\begin{align*} 
  & | J_{bad} \cap \{ x \in J: [a(t,s)u(\varphi(x))] \in \cM \} | \\ 
  =& \sum_{I \text{ bad}} |I \cap  \{ x \in J: [a(t,s)u(\varphi(x))] \in \cM \}| \\ 
  \le& \sum_{I \text{ bad}} 8  e^{-2 s^\ast - t^\ast} \\
  \le& |J| (10^6 c^{-1} e^{-2 s^\ast - t^\ast })^{-1} 8  e^{-2 s^\ast - t^\ast} \le 10^{-5} c |J|.
\end{align*}
Therefore, we have 
\[  | J_{good} \cap \{ x \in J: [a(t,s)u(\varphi(x))] \in \cM \} | \ge 0.9 c |J|. \]
Note that 
\begin{align*}
  & | J_{good} \cap \{ x \in J: [a(t,s)u(\varphi(x))] \in \cM \} | \\
  \le& |\{I: I \text{ is good}\}| 10^6 c^{-1} e^{-2 s^\ast - t^\ast }.
\end{align*}
It implies that 
\[ |\{I: I \text{ is good}\}| \ge 0.9 c |J| \times 10^{-6} c e^{2 s^\ast + t^\ast} \ge 10^{-7} c^2 |J| e^{2s^\ast + t^\ast}. \]
Then 
\begin{align*}
  |B^\ast(t,s)| &= |\{I: I \text{ is good}\}| \times 2 e^{-2s -t} \\
                &=  10^{-7} c^2 |J| e^{2s^\ast + t^\ast}  \times 2 e^{-2s -t} \\
                &\ge 10^{-7} c^2 |J| e^{-3\kappa}.
\end{align*}
Noting that $e^{-3\kappa} \gg \epsilon^3 (s+t)^{-2}$, we get
\begin{equation}
\label{equ:B-ast-t-s-size}
 |B^{\ast}(t,s)| \gg |J| \epsilon^3 (s+t)^{-2}
\end{equation}
where the implied constant is independent of $J$. Therefore
\begin{align*}
\sum_{(t,s) \in \mathcal{R}} |B^{\ast} (t,s)| &\gg  \sum_{t+s \ge T_1} |J| \epsilon^3 (s+t)^{-2}
 =  \sum_{\ell=T_{1}}^{\infty} \sum_{t+s = \ell} |J| \epsilon^3 \ell^{-2} \gg_{\eps,J} \sum_{\ell =T_{1}}^{\infty} \ell^{-1} = \infty.
\end{align*}
This confirms \eqref{equ:borel-cantelli-divergent}.

We turn our attention towards \eqref{equ:borel-cantelli-independent}. For $(t,s), (t',s') \in \mathcal{R}$, let us estimate $|B^{\ast} (t,s) \cap B^{\ast}(t',s') |$. Recall that for $(t,s ) \in \mathcal{R}$ we have $c_1 s \leq  t  \leq c_2 s$. Put $c_3 = (c_2 - c_1)/4$. Without loss of generality, we may assume that $2s' + t'  > 2 s + t$. We will consider the following two cases separately:
\begin{enumerate} [(i)]
	\item $(2s' + t') - (2s + t) \geq c_3 (2s + t) $.
	\item $ (2 s' + t') - (2s + t) < c_3 (2s + t) $.
\end{enumerate}
\par Let us take care of the first case. Consider a small interval $I_1 \subset B^{\ast}(t,s)$. We compute that
\[
|I_1| = 2e^{-2s-t}.
\]
Let us count how many small intervals from $B^{\ast}(t',s')$ are contained in $I_1$. By Definition \ref{def:subset-B-ast-t-s}, every small interval $I'_1$ from $B^{\ast}(t', s')$ has length $2e^{-2s' - t'}$, and is contained in an interval $I'$ of length $2 \epsilon^{-3} (s'+t')^{2} e^{-2s' -t'}$ which does not intersect any other small intervals from $B^{\ast}(t',s')$. Therefore $I_1$ intersects at most
\[
2+ \lfloor e^{-2s -t} (\epsilon^{-3} (s'+t')^{2} e^{-2s' -t'})^{-1} \rfloor
\]
small intervals from $B^{\ast}(t', s')$, and note that the second term dominates because we are in Case (i). This implies that
\begin{align*}
 |I_1 \cap B^{\ast}(t',s')| &  \le  (2 +  e^{-2s -t} (\epsilon^{-3} (s'+t')^{2} e^{-2s' -t'})^{-1} ) 2 e^{-2s' -t'} \\
& \le 2 |I_1| \epsilon^3 (s'+ t')^{-2} \ll  |J|^{-1}  |I_1| \cdot |B^{\ast}(t',s')|,
\end{align*}
where the last inequality comes from \eqref{equ:B-ast-t-s-size}. Now
\begin{align*}
|B^{\ast}(t,s)\cap B^{\ast}(t',s') | & = \sum_{I_1 \subset B^{\ast}(t,s)} |I_1 \cap B^{\ast}(t',s') |
\\ &\ll  |J|^{-1} \sum_{I_1 \subset B^{\ast}(t,s)} |I_1| \cdot |B^{\ast}(t',s')| \\
 & \le  |J|^{-1} |B^{\ast}(t,s) | \cdot | B^{\ast}(t',s') |.
\end{align*}
We conclude that if $(2s' + t') - (2s + t) \geq c_3 (2s + t)$ then
\begin{equation}
\label{equ:independence-easy-case}
|B^{\ast}(t,s)\cap B^{\ast}(t',s') |  \ll   |J|^{-1} |B^{\ast}(t,s) | \cdot | B^{\ast}(t',s') |
\end{equation}
where the implied constant is independent of $J$.

\bigskip
We now examine Case (ii). Since $c_1 s' \leq  t' \leq c_2 s'$ and $c_3 =(c_2 - c_1)/4 $, we deduce that
\[ |s' -s| \leq (c_2 - c_1) s \text{ and } |t' - t| \leq (c_2 - c_1) s. \]
Define $g = a(t' - t, s' -s)$, noting that $a(t' , s') = g a(t,s)$ and $\|g\| \leq e^{2(c_2- c_1) s}$.

For $\theta \in (0,1/2)$, let $D(\theta ) = X \setminus K_{\theta}$. By \cite[Proposition 7.1]{Kleinbock-Margulis-Log-Laws},
\[
\mu_G(D(\theta)) \asymp \theta^3.
\]
Let $m' = m'(G) \in \bN$ be a constant to be determined in due course. By the correct version of \cite[Lemma 4.2]{Kleinbock-Margulis-Log-Laws}, namely \cite[Theorem 1.1]{Kleinbock-Margulis-Log-Laws-Erratum}, applied with distance-like (DL) function $\Del$ from \eqref{DeltaDefinition}, for each $\ell \in \bN$ there exists a function $F_\ell \in C^{\infty}_b (X)$ such that---with the same implicit constants for all $\ell$---
\begin{enumerate}
	\item $0 \leq F_\ell(x) \leq 1$, and $F_\ell (x) = 1$ for $x \in D(\epsilon \ell^{-2/3})$;
	\item $\int_X F_\ell \:\dd \mu_G \ll \mu_G (D(\epsilon l^{-2/3})) \ll \epsilon^3 \ell^{-2}  $;
	\item $  \| F_\ell \|_{2,m'} \ll \mu_G (D(\epsilon l^{-2/3})) \ll \epsilon^{3} \ell^{-2}  $.
\end{enumerate}
(The \emph{Sobolev norm} $\| \cdot \|_{2,m'}$ is defined by $ \| h \|_{2,m'} = \displaystyle \sum_{ \ord(D) \le m'} \| D h \|_2$. For further details, see \cite{Kleinbock-Margulis-Log-Laws-Erratum}.) Moreover, by examining the construction, we also have that $\|F_\ell\|_{C^m_b} \ll 1$, where $m$ is as in Theorem \ref{thm:effective-ratner-main-result} (with the same implicit constant for all $\ell$).

We have
\[ |B^{\ast}(t, s) \cap B^{\ast}(t',s') | \leq \int_J F_\ell(a(t,s)u(\varphi(x))) F_{\ell'} ( a(t',s')u(\varphi(x)) ) \dd x, \]
where $\ell = t+s$ and $\ell' = t' + s'$. Thus, to get the desired upper bound on
\[ |B^{\ast}(t, s) \cap B^{\ast}(t',s') |, \]
it suffices to bound
\[
\int_J F_\ell(a(t,s)u(\varphi(x))) F_{\ell'} ( a(t',s')u(\varphi(x)) ) \dd x
\]
from above. For $x \in X$, we write $F(x) = F_\ell(x)F_{\ell'}(g x)$, where $g = a(t' - t, s' - s)$ as before. The above integral is equal to
\[  \int_J F(a(t,s)u(\varphi(x))) \dd x. \]

Theorem \ref{thm:effective-ratner-main-result} gives
\begin{align*}   \left| \frac{1}{|J|}\int_J F(a(t,s) u(\varphi(x))) \dd x - \int_X F(x) \dd \mu_G(x) \right| & \ll \frac{1}{|J|} e^{-\eta t} \|F\|_{C^m_b} \\& \leq \frac{1}{|J|} e^{-\eta c_1 s} \|F\|_{C^m_b}.
\end{align*}
By the product rule, we have
\[
\|F\|_{C^m_b} \ll \|F_\ell \|_{C^m_b} \|F_{\ell'} \circ g \|_{C^m_b}.
\]
Our choice of $\{F_\ell\}$ ensures that
\[
\|F_\ell(\cdot)\|_{C^m_b} \ll 1
\]
and
\[
\|F_{\ell'} \circ g\|_{C^m_b} \ll \|g\|^m \|F_{\ell'}\|_{C^m_b} \ll e^{2m(c_2 - c_1 ) s}.
\]
Therefore
\[
\|F\|_{C^m_b} \ll e^{2m (c_2 - c_1) s}.
\]
We may choose $c_1$ sufficiently close to $c_2$ in order to ensure that
\[
2m (c_2 - c_1)  < \eta c_1/2,
\]
and so
\begin{equation}
\label{equ:critical-line-integral}
\left| \int_J F(a(t,s) u(\varphi(x))) \dd x - |J|\int_X F(x) \dd \mu_G(x) \right| \ll  e^{-\eta c_1 s/2}.
\end{equation}

Next, we consider
\[ |J|\int_X F(x) \dd \mu_G(x) = |J|\int_X F_\ell(x) F_{\ell'} (a(t'-t, s' -s) x) \dd \mu_G(x). \]
We need to estimate
\[
\left| |J|\int_X F_\ell(x) F_{\ell'} (a(t'-t, s' -s) x) \dd \mu_G(x) - |J| \int_X F_\ell \dd \mu_G \int_X F_{\ell'} \dd \mu_G \right|.
\]
By the exponential mixing property of the action of $a(t,s)$, see \cite[Corollary 3.5]{Kleinbock-Margulis-Log-Laws} and \cite[Equation (EM)]{Kleinbock-Margulis-Log-Laws-Erratum}, there exist constants $\eta_8 >0$ and $m' \in \bN$ such that
\[
\begin{array}{cl}
     & \left| \int_X F_\ell(x) F_{\ell'} (a(t'-t, s' -s) x) \dd \mu_G(x) - \int_X F_\ell \dd \mu_G \int_X F_{\ell'} \dd \mu \right| \\
 \ll &  e^{-\eta_8 \max( |t-t'|, |s - s'| )}  \| F_\ell \|_{2,m'} \| F_{\ell'} \|_{2,m'} \\
 \ll & e^{-\eta_8 \max( |t-t'|, |s - s'| )} \sqrt{\epsilon^3 \ell^{-2} } \sqrt{ \epsilon^3 (\ell')^{-2} }.
\end{array}
\]
(Note that we have now specified $m'$.) The last inequality follows from the third property of $\{ F_\ell \}$. Since $\ell \asymp \ell'$ and $\epsilon^3 \ell^{-2} \asymp |J|^{-1} |B^{\ast}(t,s)| $, we have
\begin{equation}
\label{equ:critical-exponential-mixing}
\begin{array}{cl}
 & \left| |J|\int_X F_\ell(x) F_{\ell'} (a(t'-t, s' -s) x) \dd \mu_G(x) - |J|\int_X F_\ell \dd \mu_G \int_X F_{\ell'} \dd \mu_G \right| \vspace{1mm}\\
  \ll & |J| e^{-\eta_8 \max ( |t-t'|, |s-s'|)}|J|^{-1} |B^{\ast}(t,s)| = e^{-\eta_8 \max ( |t-t'|, |s-s'|)} |B^{\ast}(t,s)| .
\end{array}
\end{equation}

Combining \eqref{equ:critical-line-integral} and \eqref{equ:critical-exponential-mixing} gives
\[
\begin{array}{cl}
   & \int_J F_\ell(a(t,s)u(\varphi(x))) F_{\ell'} ( a(t',s')u(\varphi(x)) ) \dd x - |J|\int_X F_\ell \dd \mu_G \int_X F_{\ell'} \dd \mu_G \\
   \ll & e^{-\eta c_1 s/2} + e^{-\eta_8 \max ( |t-t'|, |s-s'|)}|B^{\ast}(t,s)|,
\end{array}	
\]
which implies that
\begin{align*}
 |B^{\ast}(t,s)\cap B^{\ast}(t',s') | &\ll |J|\int_X F_\ell \: \dd \mu_G \int_X F_{\ell'} \: \dd \mu_G
\\&\qquad+ e^{-\eta c_1 s/2} + e^{-\eta_8 \max ( |t-t'|, |s-s'|)} |B^{\ast}(t,s)|.
\end{align*}
Since $\int_X F_\ell \: \dd \mu_G  \ll |J|^{-1} |B^{\ast}(t,s)|$ and $|J| \ll 1$, we now have
\begin{align}
\notag
|B^{\ast}(t,s)\cap B^{\ast}(t',s') | &\ll |J|^{-1} |B^{\ast}(t,s) | \cdot |B^{\ast}(t',s')|
+ |J|^{-1}e^{-\eta_9 s}  \\ \label{equ:critical-independence} & \qquad + e^{-\eta_8 \max ( |t-t'|, |s-s'|)}  |J|^{-1} |B^{\ast}(t,s)|,
\end{align}
where $\eta_9 = \eta c_1/2$.\\

For $(t,s) \in \mathcal{R}$, let us denote
\begin{align*}
\mathcal{R}(t,s) &:=\{ (t',s'): 2s' +t' \geq 2 s + t \}, \\
\mathcal{R}_1(t,s)& := \{ (t',s') \in \mathcal{R}(t,s): (2s' + t') - (2s + t) \geq c_3 (2s+t)  \},
\end{align*}
and $\mathcal{R}_2(t,s) := \mathcal{R}(t,s)\setminus \mathcal{R}_1(t,s)$. For $(t',s') \in \mathcal{R}_1(t,s)$ we have \eqref{equ:independence-easy-case}, and for $(t',s') \in \mathcal{R}_2(t,s)$ we have \eqref{equ:critical-independence}. Note that all implicit constants in the estimates above are independent of $J$, $(t,s)$ and $(t',s')$. Recall that $V(N) = \{ (t,s) \in \bN^2: t+s \leq N\}$. Therefore
\begin{equation*}
\begin{array}{cl}
   & \displaystyle \sum_{(t,s) \in \mathcal{R}\cap V(N)} \sum_{(t',s') \in \mathcal{R}\cap V(N)} |B^{\ast} (t,s)\cap B^{\ast}(t',s')|  \\
    \ll &  |J|^{-1} \displaystyle \Bigl( \sum_{(t,s) \in \mathcal{R}\cap V(N)} |B^{\ast}(t,s)| \Bigr)^2  + |J|^{-1}\sum_{(t,s) \in \mathcal{R} \cap V(N)}  s^2 e^{-\eta_9 s} \\
         & + \:   |J|^{-1} \displaystyle  \sum_{(t,s) \in \mathcal{R}\cap V(N)} | B^{\ast}(t,s) | \sum_{(t',s') \in \mathcal{R}_2(t,s) \cap V(N)} e^{-\eta_8 \max ( |t-t'|, |s-s'|)} \\
      \ll &  |J|^{-1} \displaystyle  \Bigl( \sum_{(t,s) \in \mathcal{R} \cap V(N)} |B^{\ast}(t,s)| \Bigr)^2  +  |J|^{-1}\sum_{s \geq T_1} s^3 e^{-\eta_9 s} \\
& + \:  |J|^{-1}  \displaystyle \sum_{(t,s) \in \mathcal{R}\cap V(N)} | B^{\ast}(t,s) | \sum_{k=1}^{\infty} k^2 e^{-\eta_8 k}.
\end{array}
\end{equation*}

Note from \eqref{equ:B-ast-t-s-size} that if $N \ge \max \{T_1, e^{|J|^{-1}}\}$ then
\[
\sum_{(t,s) \in \mathcal{R}\cap V(N)} |B^{\ast}(t,s)|  \gg |J| \log N \gg 1.
\]
Now
\[
 \sum_{s \geq T_1} s^3 e^{-\eta_9 s} \ll 1 \ll \Bigl( \sum_{(t,s) \in \mathcal{R}\cap V(N)} |B^{\ast}(t,s)| \Bigr)^2
\]
and
\[ \sum_{k=1}^{\infty} k^2 e^{-\eta_8 k} \ll 1 \ll \sum_{(t,s) \in \mathcal{R}\cap V(N)} |B^{\ast}(t,s)| . \]
We obtain \eqref{equ:borel-cantelli-independent}, which completes the proof of Proposition \ref{prop:limsup}, and hence of Theorem \ref{thm:mda-result}.
\end{proof}

\section{The convergence theory}
\label{ConvergenceTheory}

In this section, we establish Theorems \ref{ConvergenceResult} and \ref{ConvergenceVariant}. We focus our attention on Theorem \ref{ConvergenceResult}, and explain at the end how the proof can be modified to give Theorem \ref{ConvergenceVariant}. We follow \cite[\S 4]{BV2007}, with $\cC$ being a fixed segment of $L_{a,b}$ instead of an arc. Recall that $L_{a, b} := \{ (\alpha, \beta) \in \bR^2: \alpha = a \beta + b \}$. With $I \subset \bR$ a fixed, bounded interval, let us explicitly write $\cC = \{ (a \bet + b, \bet): \bet \in I \}$. Let $\ome^* = \ome^*(a,b)$ be the dual exponent of $(a,b)$. Note that
\[
2 \le \ome^* < 5,
\]
where the first inequality is Dirichlet's approximation theorem (see \cite{KW}) and the second is hypothesised.

The key ingredient is certain structural data concerning the \emph{dual Bohr set}
\[
B = B(\del; Q; a,b) := \{ (p_2, q, p_1) \in \bZ^3: |p_2|, |q| \le Q, | p_2 a + qb - p_1| \le \del \},
\]
where $Q \in \N$, and $a, b, \delta \in \R$ with $\delta > 0$. Specifically, we will show that $B$ is tightly contained within a \emph{generalised arithmetic progression}
\[
P = P(\bv_1, \bv_2, \bv_3; N_1, N_2, N_3) := \{  n_1 \bv_1 + n_2 \bv_2 + n_3 \bv_3 : n_i \in \bZ, |n_i| \le N_i \},
\]
for some $\bv_1, \bv_2, \bv_3 \in \bZ^3$ and some $N_1, N_2, N_3 \in \bN$.

\begin{lemma} [Outer structure of dual Bohr sets] \label{outer}
Assume that $Q \ge Q_0(a,b)$, where $Q_0(a,b)$ is a suitably large constant, and
\[
\frac Q{1000} > \del \gg Q^{-1/2} (\log Q)^{-3/2}.
\]
Then there exist $N_1, N_2, N_3 \in \bN$, and linearly independent $\bv_1, \bv_2, \bv_3 \in \bZ^3$, such that
\begin{equation}
\label{StructuralBohr}
B \subseteq P, \qquad \#P \ll_{a,b} \del Q^2.
\end{equation}
In particular, we have
\begin{equation}
\label{CountingBohr}
\#B \ll_{a,b} \del Q^2.
\end{equation}
\end{lemma}

\begin{proof} Observe that $B$ is the set of lattice points in the region
\[
\cB := \{ (p_2, q, p_1) \in [-Q,Q]^2 \times \bR: |p_2 a + qb - p_1| \le \del \}.
\]
Put
\[
\lam = (\del Q^2)^{1/3}, \qquad \cS = \lam^{-1} \cB.
\]
Let $\lam_1 \le \lam_2 \le \lam_3$ be the reduced successive minima \cite[Lecture X]{Sie1989} of the symmetric convex body $\cS$. Corresponding to these are vectors $\bv_1, \bv_2, \bv_3 \in \bZ^3$ whose $\bZ$-span is $\bZ^3$, and for which $\bv_i \in \lam_i \cS$ ($i=1,2,3$). By the first finiteness theorem \cite[Lecture X, \S 6]{Sie1989}, we have
\begin{equation} \label{FFT}
\lam_1 \lam_2 \lam_3 \asymp \vol(\cS)^{-1} \asymp 1,
\end{equation}
and in fact
\[
\lam_1 \lam_2 \lam_3 \le \frac{27}{\vol(\cS)} = \frac{27}{8},
\]
so $\lam_1\le 3/2$.

Next, we bound $\lam_1$ from below. We know that
\[
\bv_1 \in \lam_1 \cS = \frac{\lam_1}{\lam} \cB
\]
has integer coordinates, so with $\bv_1 = (x,y,z) \ne \boldsymbol 0$ we have
\[
|x|, |y| \le \frac{\lam_1}{\lam} Q, \qquad | xa + yb - z | \le \frac{\lam_1}{ \lam} \del \le \frac32 (\del/Q)^{2/3} < \frac12.
\]
Hence $(x,y) \ne (0,0)$ and, from the definition of the dual exponent, we have
\[
\frac{\lam_1}{ \lam}  \del  \ge |xa + yb - z| = \langle xa + yb \rangle\gg_{a,b,\eps} \Bigl(\frac{\lam_1}{\lam} Q \Bigr)^{-\ome^*-\eps}
\]
for any $\eps > 0$, which rearranges to
\[
\lam_1 \gg \lam (\del Q^{\ome^*+\eps})^{-(1+\ome^*+\eps)^{-1}}.
\]

This enables us to bound $\lam_3$ from above: from \eqref{FFT}, we have $\lam_3 \ll \lam_1^{-2}$. In particular, we now know that
\begin{align*}
\frac{\lam}{\lam_3} &\gg \lam^3  (\del Q^{\ome^*+\eps})^{-2/(1+\ome^*+\eps)} =
\del Q^2 (\del Q^{\ome^*+\eps})^{-2/(1+\ome^*+\eps)} \\
&\gg (Q^{-1/2} (\log Q)^{-3/2})^{1 - 2(1+\ome^*+\eps)^{-1}} Q^{2 /(1+\ome^*+\eps)} \gg Q^{\frac3{1+\ome^*+\eps} - \: \frac12 - \eps}.
\end{align*}
Therefore $\lam \ge \lam_3$, since $\ome^* < 5$, and since $\eps > 0$ is arbitrary. We now specify our length parameters
\[
N_i =  \Bigl \lfloor \frac{C \lam} {\lam_i} \Bigr \rfloor \ge C \qquad (i = 1,2,3),
\]
where $C \ge C_0(a,b) \in \bN$ is a large constant. Observe that
\[
\#P \ll N_1 N_2 N_3 \ll \frac{\lam^3}{\lam_1 \lam_2 \lam_3} \ll \del Q^2.
\]

Our final task is to show that $B \subseteq P$. Let $\bx \in B$. Since $\bv_1, \bv_2, \bv_3$ generate $\bZ^3$, there exist $n_1, n_2, n_3 \in \bZ$ such that
\[
\bx = n_1 \bv_1 + n_2 \bv_2 + n_3 \bv_3.
\]
Let $M = (\bv_1, \bv_2, \bv_3) \in \mathrm{GL}(3, \bZ)$, and for $i=1,2,3$ let $M_i$ be the matrix obtained by replacing the $i$th column of $M$ by $\bx$. Cramer's rule gives
\[
|n_i| = |\det(M_i)|.
\]

Let $\{ i, j, r \} = \{ 1, 2, 3 \}$.
As $\cS$ is convex and contains 
\[
\bzero, \quad
\lam^{-1} \bx, \quad
\lam_j^{-1} \bv_j, \quad
\lam_r^{-1} \bv_r,
\]
it must contain the $[0,1]$-span of the vectors
\[
(3 \lam)^{-1} \bx, \quad
(3 \lam_j)^{-1} \bv_j, \quad
(3 \lam_r)^{-1} \bv_r,
\]
which is a parallelopiped $\cP$. Now
\begin{align*}
1 &\gg \vol(\cS) \ge \vol(\cP) = |\det((3 \lam)^{-1} \bx, 
(3 \lam_j)^{-1} \bv_j, 
(3 \lam_r)^{-1} \bv_r)| \\
&\gg \frac{|\det(M_i)|}
{\lam \lam_j \lam_r} = 
|n_i| \frac{\lam_i}{\lam \lam_i \lam_j \lam_r}
= |n_i| \frac{\lam_i}{\lam \lam_1 \lam_2 \lam_3} \gg |n_i| \frac{\lam_i}{\lam},
\end{align*}
so
\[
n_i \ll \lam/\lam_i.
\]
As $C$ is large, we have $|n_i| \le N_i$ for all $i$, so $\bx \in P$.
\end{proof}

\begin{remark} After showing that $\lam \ge \lam_3$ in the proof above, we could have cited a general counting result such as \cite[Proposition 2.1]{BHW} to obtain \eqref{CountingBohr}. We thank an anonymous referee for kindly pointing this out. We have left the full statement and proof intact, as we believe the structural assertion \eqref{StructuralBohr} to have independent interest. We envision applications to (1) the dual theory of multiplicative approximation on fibres, and (2) the discrepancy theory of Kronecker sequences.
\end{remark}

As in \cite[\S 4]{BV2007}, we may assume that 
\begin{equation} \label{BVassumption}
q^{-1} (\log q)^{-3} < \psi(q) < q^{-1} (\log q)^{-1}
\end{equation}
for all sufficiently large $q$. For $t \in \bN$ and $m \in \bZ$, denote by $N(t,m)$ the number of integer triples $(q,p_1,p_2)$ with $2^t \le q < 2^{t+1}$ for which there exists $\bet \in I$ such that
\[
\Bigl| a \bet + b - \frac{p_1}q \Bigr| < \frac{2^m \sqrt{2 \psi(2^t)}}{2^t}
\]
and
\[
\Bigl| \bet - \frac{p_2}q \Bigr| < \frac{2^{-m} \sqrt{2 \psi(2^t)}}{2^t}.
\]
By the triangle inequality, we have
\begin{equation} \label{delta}
| p_2 a + qb - p_1| < 3(1+|a|) 2^{|m|} \sqrt{\psi(2^t)}=: \del.
\end{equation}
Let $C_I$ be a large positive constant depending only on $I$. Applying Lemma \ref{outer} with $Q = (1+ |a|) 2^t C_I$, we obtain the following.

\begin{cor} [Counting rational points near general lines] \label{counting} If $t$ is sufficiently large and
\begin{equation} \label{tdelta}
2^{|m|} t \sqrt{ \psi(2^t) } \le 1
\end{equation}
then
\[
N(t,m) \ll_{a,b,I} 2^{|m|} 2^{2t} \sqrt{\psi(2^t)}.
\]
\end{cor}

\noindent This estimate matches \cite[Equation (35)]{BV2007}, and the rest of the proof in 
\cite[\S 4]{BV2007} applies almost verbatim in the present context of Theorem \ref{ConvergenceResult}.\\

If we slightly alter our circumstances, then there is a Fourier-analytic way to bound the cardinality of the dual Bohr set. Suppose $\ome^\times(a,b) < 4$. Consider the case $(d,n) = (1,2)$ of \cite[Theorem 8]{HuangLiu}, which we state below.

\begin{lemma} [Huang and Liu]
Let $\alp_1, \alp_2, \ome \in \bR$ with
\[
\langle q \alp_1 \rangle
\langle q \alp_2 \rangle \gg
|q|^{-\ome}
\qquad (0 \ne q \in \bZ).
\]
Let $Q \in \bN$ and $\del \in (0,1/2]$, and let $\eps > 0$. Then
\begin{align*}
&|\# \{ (q_1, q_2) \in \bN^2:
\langle q_1 \alp_1 + q_2 \alp_2
\rangle < \del, \quad
q_1, q_2 \le Q \} - 2 \del Q^2|
\\
&\le \eps \del Q^2 + 
O_\eps(\del^{1-\ome} \log^2 (1/\del)).
\end{align*}
\end{lemma}

Applying this to \eqref{delta} with \[
(\alp_1, \alp_2) = (\pm a, b),
\qquad
\ome \in (\ome^\times(a,b), 4),
\qquad
Q \asymp 2^t,
\qquad
\eps = 1,
\]
and bounding by $2^t$ the number of solutions with $p_2 = 0$, furnishes
\begin{equation}
\label{NtmBound}
N(t,m) \ll \del 2^{2t} + \del^{1-\ome} \log^2(1/\del) + 2^t.
\end{equation}
Note that $\del \asymp 2^{|m|} \sqrt{\psi(2^t)}$. By \eqref{BVassumption} and \eqref{tdelta}, we have 
\[
\frac1{t^{3/2} 2^{t/2}} \ll \del \ll \frac1t,
\]
so $2^t \ll \del 2^{2t}$. As $\ome < 4$, now
\[
\del^{1-\ome} \log^2(1/\del) \ll \del 2^{2t}.
\]
Inserting this into \eqref{NtmBound}, we arrive at the conclusion of Corollary \ref{counting}, with the hypothesis 
$
\ome^\times(a,b) < 4
$
in lieu of the hypothesis 
$\ome^*(a,b) < 5$. We thus obtain Theorem \ref{ConvergenceVariant}.


\section{Logarithm laws for lines in homogeneous spaces}
\label{sec-proof-log-law-lines}

In this section, we prove Corollary \ref{cor:log-law-line}. Here we recall \eqref{DeltaDefinition} and \eqref{DeltaFormula}.
\begin{proof}[Proof of Corollary \ref{cor:log-law-line}]
\par Let us take $o = [e]$. For $x \in [L_{a,b}]$, we only consider $(t,s) \in Q_1$ such that $\Delta(a(t,s)x) \geq L$, since the other values of $a(t,s)x$ are trapped in a fixed compact subset and do not contribute to the limit.
\par By our discussion in \S \ref{subsec-DA-to-HD}, Theorem \ref{thm:mda-result} implies that for almost every $ x \in [L_{a,b}]$ we have
\[ \Delta(a(t,s) x) \ge - \log ( (s+t)^{-2/3}) = \frac{2}{3} \log (s+t) \]
for some subsequence of $(t, s) \in \R_{>0}^2$ with $s+t \to +\infty$. Therefore
\[ \limsup_{\substack{\bt \in Q_1 \\ \|\vv t\| \to \infty}} \frac{\Delta(a(\vv t)x)}{\log \|\vv t\|} \geq \frac{2}{3}.\]
In the same way, Theorems \ref{ConvergenceResult} and \ref{ConvergenceVariant} imply that for almost every $x \in [L_{a,b}]$ we have
\[ \limsup_{\substack{\bt \in Q_1 \\ \|\vv t\| \to \infty}} \frac{\Delta(a(\vv t)x)}{\log \|\vv t\|} \le \kappa \]
for any $\kappa > \frac{2}{3}$. This proves \eqref{equ:log-law-1}.
By \eqref{DeltaFormula}, there exists constants $C_2 > C_1 > 0$ such that 
 \begin{equation}
 \label{equ:chain-of-inequality}
  C_1 \frac{\Delta(a(\vv t)x)}{\log \|\vv t\|} \le \frac{d(a(\vv t)x,[e])}{\log \|\vv t\|} \le C_2 \frac{\Delta(a(\vv t)x)}{\log \|\vv t\|}.  
  \end{equation}
We complete the proof by taking $\displaystyle \limsup_{\substack{\bt \in Q_1 \\ \|\vv t\| \to \infty}}$ in the chain of inequalities above.
\end{proof}
\begin{remark} \label{KhintchineRemark}
Using Khintchine's theorem \cite[\S 1.2.2]{BRV} in place of Theorem~\ref{thm:mda-result}, we can still obtain a positive lower bound on 
\[ \limsup_{\substack{\bt \in Q_1 \\ \|\vv t\| \to \infty}} \frac{d(a(\vv t)x,[e])}{\log \|\vv t\|}.  \]
Indeed, Khintchine's theorem gives 
\[ \liminf_{n \to \infty} n (\log n) \langle n\alpha \rangle = 0 \]
for almost all $\alpha \in \R$, so for almost all $(\alpha, \beta) \in L_{a, b}$ we have \eqref{Khintchine}. By our discussion in \S \ref{subsec-DA-to-HD}, this implies that for almost every $x \in [L_{a,b}]$ we have 
\[ \limsup_{\substack{\bt \in Q_1 \\ \|\vv t\| \to \infty}} \frac{\Delta(a(\vv t)x)}{\log \|\vv t\|} \geq \frac{1}{3}.\]
We obtain a positive constant lower bound on 
\[ \limsup_{\substack{\bt \in Q_1 \\ \|\vv t\| \to \infty}} \frac{d(a(\vv t)x,[e])}{\log \|\vv t\|}\] 
by taking $\displaystyle \limsup_{\substack{\bt \in Q_1 \\ \|\vv t\| \to \infty}}$  in the left inequality of \eqref{equ:chain-of-inequality}. 

As mentioned in Remark \ref{quadrant}, this argument, though it does not require Theorem \ref{thm:mda-result}, gives rise to a poorer lower bound $E_1$.
\end{remark}

\bibliography{reference}{}
\bibliographystyle{alpha}

\end{document}